\documentclass[11pt,reqno]{amsart}
\usepackage[dvipdfm,a4paper,top=30mm,bottom=25mm,left=25mm,right=25mm]{geometry}
\usepackage{amsmath,amssymb,mathtools}
\usepackage{comment,mathrsfs}

\usepackage[dvipdfm,colorlinks=true,linkcolor=blue,citecolor=blue]{hyperref}

\newtheorem{theorem}{Theorem}[section]
\newtheorem{lemma}[theorem]{Lemma}
\newtheorem{proposition}[theorem]{Proposition}
\newtheorem{corollary}[theorem]{Corollary}

\theoremstyle{definition}
\newtheorem{definition}[theorem]{Definition}

\theoremstyle{remark}
\newtheorem{remark}[theorem]{Remark}

\numberwithin{equation}{section}

\allowdisplaybreaks[4]

\DeclareMathOperator*{\essinf}{ess\,inf}
\def\rnum#1{\expandafter{\romannumeral #1}} 
\def\Rnum#1{\uppercase\expandafter{\romannumeral #1}}

\begin{document}

\title[Nonlinear Schr\"odinger equation with a potential]{Equivalence of conditions on initial data below the ground state to NLS with a repulsive inverse power potential}


\author{Masaru Hamano}
\address{Department of Mathematics, Graduate School of Science and Engineering Saitama University, Shimo-Okubo 255, Sakura-ku, Saitama-shi, Saitama 338-8570, Japan}
\email{m.hamano.733@ms.saitama-u.ac.jp}

\author{Masahiro Ikeda}
\address{Department of Mathematics, Faculty of Science and Technology, Keio University, 3-14-1 Hiyoshi, Kohoku-ku, Yokohama, 223-8522, Japan/Center for Advanced Intelligence Project, Riken, Japan}
\email{masahiro.ikeda@riken.jp/masahiro.ikeda@riken.jp}







\begin{abstract}
In this paper, we consider the nonlinear Schr\"odinger equation with a repulsive inverse power potential.
First, we show that some global well-posedness results and ``blow-up or grow-up'' results below the ground state without the potential.
Then, we prove equivalence of the conditions on the initial data below the ground state without potential.
We note that recently, we established existence of a radial ground state and characterized it by the virial functional for NLS with a general potential in two or higher space dimensions in \cite{HamIkeISAAC}.
Then, we also prove a global well-posedness result and a "blow-up or grow-up" result below the radial ground state with a repulsive inverse power potential obtained in \cite{HamIkeISAAC}.
\end{abstract}

\maketitle

\tableofcontents


\section{Introduction}\label{Introduction}

\subsection{Background}

In this paper, we consider the following nonlinear Schr\"odinger equation with a repulsive inverse power potential:
\begin{equation}
\tag{NLS$_{\gamma}$}\label{NLS}
\begin{cases}
&\hspace{-0.4cm}\displaystyle{i\partial_tu+\Delta_\gamma u=-|u|^{p-1}u,\quad (t,x)\in\mathbb{R}\!\times\!\mathbb{R}^d},\\
&\hspace{-0.4cm}\displaystyle{u(0,\,\cdot\,)=u_0\in H^1(\mathbb{R}^d)},
\end{cases}
\end{equation}
where $d \geq 1$, $2_\ast < p+1 < 2^\ast$,
\begin{equation*}
2_\ast-1 := 1+\frac{4}{d}, \ \ \ 2^\ast-1:=
\begin{cases}
\hspace{-0.4cm}&\displaystyle{\infty,\qquad\hspace{1.21cm}(d=1,2),}\\
\hspace{-0.4cm}&\displaystyle{1+\frac{4}{d-2},\qquad(d\geq3)},
\end{cases}
\end{equation*}
$\Delta_\gamma = \Delta-\frac{\gamma}{|x|^\mu}$, $\gamma > 0$, $0 < \mu < \min\{2,d\}$, $u = u(t,x)$ is a complex-valued unknown function, and $u_0(x) = u(0,x)$ is a complex-valued given function:\\

It can be seen in \cite[Theorem 4.3.1]{Caz03} that the Cauchy problem \eqref{NLS} is locally well-posed in the energy space $H^1(\mathbb{R}^d)$.

\begin{theorem}[Local well-posedness, \cite{Caz03}]\label{Local well-posedness}
Let $d \geq 1$, $2 < p+1 < 2^\ast$, $\gamma > 0$, and $0 < \mu < \min\{2,d\}$.
For every $u_0 \in H^1(\mathbb{R}^d)$, there exist $T_\text{max} \in (0,\infty]$, $T_\text{min} \in [-\infty,0)$, and a unique solution to \eqref{NLS} such that
\begin{align*}
	u
		\in C((T_\text{min},T_\text{max});H^1(\mathbb{R}^d)) \cap C^1((T_\text{min},T_\text{max});H^{-1}(\mathbb{R}^d)),
\end{align*}
Here, the solution $u$ does not exist beyond the interval $(T_\text{min},T_\text{max})$ and the interval $(T_\text{min},T_\text{max})$ is called maximal lifespan of $u$.
Moreover, the solution has the following blow-up alternative: If $T_\text{max} < \infty$ (resp. $T_\text{min} > -\infty$), then
\begin{align*}
	\lim_{t\nearrow T_\text{max}}\|u(t)\|_{H^1}
		= \infty,\ \ 
	\left(\text{resp. }\lim_{t\searrow T_\text{min}}\|u(t)\|_{H^1}
		= \infty\right).
\end{align*}
\end{theorem}

The $H^1$-solution to \eqref{NLS} given in Theorem \ref{Local well-posedness} preserves its mass and energy, defined respectively by
\begin{align}
	\text{(Mass) }&\ \ M[u(t)]
		:=\|u(t)\|_{L^2}^2, \notag \\
	\text{(Energy) }&\ \ E_{\gamma}[u(t)]
		:=\frac{1}{2}\|(-\Delta_\gamma)^\frac{1}{2}u(t)\|_{L^2}^2-\frac{1}{p+1}\|u(t)\|_{L^{p+1}}^{p+1}. \label{124}
\end{align}
The $H^1$-solution to \eqref{NLS} has various kinds of time behaviors by initial data.
For example, there are the following time behaviors.

\begin{definition}[Scattering, Blow-up, Grow-up, and Standing wave]\label{Time behavior}
Let $u$ be a solution to \eqref{NLS} on $(T_\text{min},T_\text{max})$.
\begin{itemize}
\item (Scattering)\\
We say that $u$ scatters in positive time (resp. negative time) if $T_\text{max} = \infty$ (resp. $T_\text{min} = -\infty$) and there exists $\psi_+ \in H^1(\mathbb{R}^d)$ (resp. $\psi_- \in H^1(\mathbb{R}^d)$) such that
\begin{align*}
	\lim_{t\rightarrow+\infty}\|u(t)-e^{it\Delta_\gamma}\psi_+\|_{H^1}
		= 0,\ \ \ 
	\left(\text{resp.} \lim_{t\rightarrow-\infty}\|u(t)-e^{it\Delta_\gamma}\psi_-\|_{H^1}
		= 0 \right).
\end{align*}
\item (Blow-up)\\
We say that $u$ blows up in positive time (resp. negative time) if $T_{\text{max}}<\infty$ (resp. $T_\text{min}>-\infty$).
\item (Grow-up)\\
We say that $u$ grows up in positive time (resp. negative time) if $T_\text{max}=\infty$ (resp. $T_\text{min}=-\infty$) and
\begin{align*}
\limsup_{t\rightarrow\infty}\|u(t)\|_{H^1}=\infty,\ \ \ 
	\left(\text{resp. }\limsup_{t\rightarrow-\infty}\|u(t)\|_{H^1}=\infty\right).
\end{align*}
\item (Standing wave)\\
We say that $u$ is standing wave if $u=e^{i\omega t}Q_{\omega,\gamma}$ for $\omega\in\mathbb{R}$, where $Q_{\omega,\gamma}$ satisfies
\begin{align}
	-\omega Q_{\omega,\gamma}+\Delta_\gamma Q_{\omega,\gamma}
		=-|Q_{\omega,\gamma}|^{p-1}Q_{\omega,\gamma}. \tag{SP$_{\omega,\gamma}$} \label{SP}
\end{align}
\end{itemize}
\end{definition}

After Kenig--Merle's work \cite{KenMer06}, the time behavior of solutions to \eqref{NLS} has been studied by using the ground state $Q_{\omega,\gamma}$.
We recall the definition of the ground state.
A set of the ground state to \eqref{SP} is defined as
\begin{align*}
	\mathcal{G}_{\omega,\gamma}
		&:= \left\{\phi \in \mathcal{A}_{\omega,\gamma}:S_{\omega,\gamma}(\phi)\leq S_{\omega,\gamma}(\psi)\text{ for any }\psi\in \mathcal{A}_{\omega,\gamma}\right\},
\end{align*}
where
\begin{align*}
	S_{\omega,\gamma}(f)
		:= \frac{\omega}{2}M[f] + E_\gamma[f]\ \ \text{ and }\ \ 
	\mathcal{A}_{\omega,\gamma}
		:= \bigl\{\psi\in H^1(\mathbb{R}^d)\setminus\{0\}:S_{\omega,\gamma}'(\psi)=0\bigr\}.
\end{align*}
It is well known that the ground state $Q_{\omega,0}$ to \eqref{NLS} with $\gamma = 0$ attains
\begin{align*}
	n_{\omega,\gamma}^{\alpha,\beta}
		:= \inf\left\{S_{\omega,\gamma}(f):f\in H^1(\mathbb{R}^d)\setminus\{0\},\ K_{\omega,\gamma}^{\alpha,\beta}(f)=0\right\}
\end{align*}
with $\gamma = 0$, where $(\alpha,\beta)$ satisfies
\begin{align}
	\alpha
		> 0,\ \ \ 
	\beta
		\geq 0,\ \ \ 
	2\alpha - d\beta
		\geq 0 \label{104}
\end{align}
and $K_{\omega,\gamma}^{\alpha,\beta}$ is defined as
\begin{align*}
	K_{\omega,\gamma}^{\alpha,\beta}(f)
		:= \mathcal{L}^{\alpha,\,\beta}S_{\omega,\gamma}(f)
		:= \left.\frac{\partial}{\partial \lambda}\right|_{\lambda=0}S_{\omega,\gamma}(e^{\alpha\lambda}f(e^{\beta\lambda}\,\cdot\,)).
\end{align*}
We note that \eqref{104} deduces
\[
	\overline{\lambda}	
		:= 2\alpha - (d-2)\beta
		\geq 2\alpha - (d-\mu)\beta
		\geq 2\alpha -d\beta
		=: \underline{\lambda},
\]
\[
	2\alpha - (d-\mu)\beta
		> 0,\ \ \ 
	(p+1)\alpha - d\beta
		> (p-1)\alpha - 2\beta
		> 0.
\]

When $\gamma = 0$, Holmer--Roudenko \cite{HolRou08} proved the following theorem for time behavior of solutions to \eqref{NLS} by using the ground state $Q_{1,0}$ to \eqref{SP} with $\omega = 1$ and $\gamma = 0$.

\begin{theorem}[Holmer--Roudenko, \cite{HolRou08}]
Let $d = 3$, $p = 3$, and $\gamma = 0$.
Let $Q_{1,0}$ be the ground state to \eqref{SP} with $\omega = 1$ and $\gamma = 0$.
Suppose that $u_0 \in H_\text{rad}^1(\mathbb{R}^3)$ satisfies
\begin{align}
	M[u_0]^{1-s_c}E_\gamma[u_0]^{s_c}
		< M[Q_{1,0}]^{1-s_c}E_0[Q_{1,0}]^{s_c}, \label{165}
\end{align}
where $s_c := \frac{d}{2} - \frac{2}{p-1}$.
\begin{itemize}
\item (Scattering)\\
If $u_0$ satisfies
\begin{align}
	\|u_0\|_{L^2}^{1-s_c}\|\nabla u_0\|_{L^2}^{s_c}
		< \|Q_{1,0}\|_{L^2}^{1-s_c}\|\nabla Q_{1,0}\|_{L^2}^{s_c}, \label{166}
\end{align}
then a solution $u$ to \eqref{NLS} scatters in both time directions.
\item (Blow-up)\\
If $u_0$ satisfies
\begin{align}
	\|u_0\|_{L^2}^{1-s_c}\|\nabla u_0\|_{L^2}^{s_c}
		> \|Q_{1,0}\|_{L^2}^{1-s_c}\|\nabla Q_{1,0}\|_{L^2}^{s_c}, \label{164}
\end{align}
then a solution $u$ to \eqref{NLS} blows up in both time directions.
\end{itemize}
\end{theorem}

Next, we introduce known results with $\gamma > 0$ for time behaviors given in Definition \ref{Time behavior}.
For blow-up, Dinh \cite{Din182} proved the following result.

\begin{theorem}[Dinh, \cite{Din182}]\label{Dinh182}
Let $d \geq 1$, $2_\ast < p+1 < 2^\ast$, $\gamma > 0$, and $0 < \mu < \min\{2,d\}$.
Let $Q_{1,0}$ be the ground state to \eqref{SP} with $\omega = 1$ and $\gamma = 0$.
We assume that $u_0$ satisfies \eqref{165}.
\begin{itemize}
\item (Global well-posedness)
If $u_0$ satisfies \eqref{166}, then $u$ exists globally in both time directions.
\item (Blow-up)
We assume $|x|u_0\in L^2(\mathbb{R}^d)$ with $d\geq 1$ or $u_0\in H_\text{rad}^1(\mathbb{R}^d)$ with $d \geq 2$ and $p \leq 5$.
If $u_0$ satisfies $E_\gamma(u_0) < 0$ or ``\eqref{164} and $E_\gamma (u_0) \geq 0$'', then $u$ blows up in both time directions.
\end{itemize}
\end{theorem}

To prove Theorem \ref{Dinh182}, Dinh \cite{Din182} used the fact:
\begin{align*}
	PW_{+,1}
		:= \{u_0 \in H^1(\mathbb{R}^d):\eqref{165}\text{ and }\eqref{166}\}
		\ \ \text{ and }\ \ 
	PW_{-,1}
		:= \{u_0 \in H^1(\mathbb{R}^d):\eqref{165}\text{ and }\eqref{164}\}
\end{align*}
are invariant with respect to time, that is, a solution $u(t)$ to \eqref{NLS} for any $t \in (T_\text{min},T_\text{max})$ belongs to the same sets as initial data by the following characterization of the ground state $Q_{1,0}$ without the potential.

\begin{proposition}[Gagliardo-Nirenberg inequality without a potential, \cite{Wei82}]\label{Gagliardo-Nirenberg inequality without a potential}
Let $d \geq 1$ and $2 < p+1 < 2^\ast$.
Then, the following inequality holds:
\begin{align*}
	\|f\|_{L^{p+1}}^{p+1}
		\leq C_\text{GN}\|f\|_{L^2}^{p+1-\frac{d(p-1)}{2}}\|\nabla f\|_{L^2}^\frac{d(p-1)}{2}
\end{align*}
for any $f\in H^1(\mathbb{R}^d)$, where $C_\text{GN}$ is the best constant and is attained by the ground state $Q_{1,0}$ to \eqref{SP} with $\omega=1$ and $\gamma=0$.
\end{proposition}

In \cite{HamIkeISAAC}, the authors gave the following characterization of the ground state $Q_{\omega,0}$ to \eqref{SP}.

\begin{proposition}[Minimization problem, \cite{HamIkeISAAC}]\label{Minimization problem}
Let $d \geq 3$, $2_\ast < p+1 < 2^\ast$, $x^\alpha \partial^\alpha V \in L^\frac{d}{2}(\mathbb{R}^d) + L^\sigma(\mathbb{R}^d)$ for some $\frac{d}{2} < \sigma < \infty$ and for any $\alpha \in (\mathbb{Z}_{\geq 0})^d$ with $|\alpha| \leq 1$, $V \geq 0$, $x\cdot\nabla V < 0$, and $2V + x\cdot\nabla V \geq 0$.
Let $(\alpha,\beta)$ satisfy \eqref{104} and $\omega > 0$.
Then, $n_{\omega,V}^{\alpha,\beta}$ is not attained and $n_{\omega,V}^{\alpha,\beta} = n_{\omega,0}^{\alpha,\beta} (= S_{\omega,0}(Q_{\omega,0}))$, where $E_V$ is defined as \eqref{124} by replacing $\frac{\gamma}{|x|^\mu}$ with $V$.
\end{proposition}

\begin{remark}\label{Remark Minimization problem}
If we replace $x^\alpha \partial^\alpha V \in L^\frac{d}{2}(\mathbb{R}^d) + L^\sigma(\mathbb{R}^d)$ for some $\frac{d}{2} < \sigma < \infty$ with $x^\alpha \partial^\alpha V \in L^\eta(\mathbb{R}^d) + L^\sigma(\mathbb{R}^d)$ for some $\frac{d}{2} < \eta \leq \sigma < \infty$, then Proposition \ref{Minimization problem} holds in $d = 2$.
If we replace $x^\alpha \partial^\alpha V \in L^\frac{d}{2}(\mathbb{R}^d) + L^\sigma(\mathbb{R}^d)$ for some $\frac{d}{2} < \sigma < \infty$ with $x^\alpha \partial^\alpha V \in L^1(\mathbb{R}^d) + L^\sigma(\mathbb{R}^d)$ for some $1 \leq \sigma < \infty$, then Proposition \ref{Minimization problem} holds in $d = 1$.
\end{remark}

Since $n_{\omega,\gamma}^{\alpha,\beta}$ is independent of $(\alpha,\beta)$, we express $n_{\omega,\gamma} := n_{\omega,\gamma}^{\alpha,\beta}$ for simplicity.\\

Proposition \ref{Minimization problem} deduces that
\begin{align*}
	PW_{+,2}
		:= \bigcup_{\omega>0}\{u_0 \in H^1(\mathbb{R}^d):S_{\omega,\gamma}(u_0)<S_{\omega,0}(Q_{\omega,0})\text{ and }K_{\omega,\gamma}^{d,2}(u_0)\geq 0\}
\end{align*}
and
\begin{align*}
	PW_{-,2}
		:= \bigcup_{\omega>0}\{u_0 \in H^1(\mathbb{R}^d):S_{\omega,\gamma}(u_0)<S_{\omega,0}(Q_{\omega,0})\text{ and }K_{\omega,\gamma}^{d,2}(u_0)<0\}
\end{align*}
are invariant with respect to time.
We note that a functional $K_{\omega,\gamma}^{d,2}$ is called virial functional and is written as
\begin{align*}
	K_{\omega,\gamma}^{d,2}(f)
		= 2\|\nabla f\|_{L^2}^2 + \mu\int_{\mathbb{R}^d}\frac{\gamma}{|x|^\mu}|f(x)|^2dx - \frac{d(p-1)}{p+1}\|f\|_{L^{p+1}}^{p+1}
		=: K_\gamma(f).
\end{align*}
If the initial data $u_0 \in H^1(\mathbb{R}^d) \cap |x|^{-1}L^2(\mathbb{R}^d)$, then a solution $u$ to \eqref{NLS} satisfies
\begin{align}
	\frac{d^2}{dt^2}\|xu(t)\|_{L^2}^2
		= 4K_\gamma(u(t)) \label{168}
\end{align}
on $(T_\text{min},T_\text{max})$ (see \cite[Proposition 6.5.1]{Caz03}).\\

Moreover, the Gagliardo--Nirenberg inequality with the potential:
\begin{align*}
	\|f\|_{L^{p+1}}^{p+1}
		\leq C_\text{GN}\|f\|_{L^2}^{p+1-\frac{d(p-1)}{2}}\|(-\Delta_\gamma)^\frac{1}{2} f\|_{L^2}^\frac{d(p-1)}{2}
\end{align*}
generates the invariant sets
\begin{align*}
	PW_{+,3}
		:= \left\{u_0 \in H^1(\mathbb{R}^d):\eqref{165}\text{ and }\|u_0\|_{L^2}^{1-s_c}\|(-\Delta_\gamma)^\frac{1}{2} u_0\|_{L^2}^{s_c}
		< \|Q_{1,0}\|_{L^2}^{1-s_c}\|\nabla Q_{1,0}\|_{L^2}^{s_c}\right\}
\end{align*}
and
\begin{align*}
	PW_{-,3}
		:= \left\{u_0 \in H^1(\mathbb{R}^d):\eqref{165}\text{ and }\|u_0\|_{L^2}^{1-s_c}\|(-\Delta_\gamma)^\frac{1}{2} u_0\|_{L^2}^{s_c}
		> \|Q_{1,0}\|_{L^2}^{1-s_c}\|\nabla Q_{1,0}\|_{L^2}^{s_c}\right\}.
\end{align*}

The following proposition unifies the sense of ``below the ground state without potential''.
For the equation with a more general potential, the proposition is shown by the authors in \cite{HamIkeISAAC}.

\begin{proposition}\label{Equivalence of ME and S<m}
Let $d \geq 1$, $2_\ast < p+1 < 2^\ast$, $\gamma > 0$, and $0 < \mu < \min\{2,d\}$.
The following two conditions are equivalent.
\begin{itemize}
\item[(1)] $\displaystyle M[u_0]^{1-s_c}E_\gamma^{1-s_c}[u_0] < M[Q_{1,0}]^{1-s_c}E_0[Q_{1,0}]^{s_c}$.
\item[(2)] There exists $\omega > 0$ such that $S_{\omega,\gamma}(u_0) < S_{\omega,0}(Q_{\omega,0}).$
\end{itemize}
\end{proposition}

Namely, the identities
\begin{align*}
	PW_{+,1} \cup PW_{-,1}
		= PW_{+,2} \cup PW_{-,2}
		= PW_{+,3} \cup PW_{-,3}
\end{align*}
hold.
We note that
\begin{align*}
	\|u_0\|_{L^2}^{1-s_c}\|\nabla u_0\|_{L^2}^{s_c}
		= \|Q_{1,0}\|_{L^2}^{1-s_c}\|\nabla Q_{1,0}\|_{L^2}^{s_c},
\end{align*}
and
\begin{align*}
	\|u_0\|_{L^2}^{1-s_c}\|(-\Delta_\gamma)^\frac{1}{2} u_0\|_{L^2}^{s_c}
		= \|Q_{1,0}\|_{L^2}^{1-s_c}\|\nabla Q_{1,0}\|_{L^2}^{s_c}
\end{align*}
never hold by the assumption \eqref{165} (see Lemma \ref{Coercivity 1}).
It is a natural question that the relation of the conditions on the initial data below the ground state without a potential: \eqref{166}, \eqref{164}, $K_\gamma(u_0) \geq 0$, $K_\gamma(u_0) < 0$,
\begin{align}
	\|u_0\|_{L^2}^{1-s_c}\|(-\Delta_\gamma)^\frac{1}{2} u_0\|_{L^2}^{s_c}
		< \|Q_{1,0}\|_{L^2}^{1-s_c}\|\nabla Q_{1,0}\|_{L^2}^{s_c}, \label{169}
\end{align}
and
\begin{align}
	\|u_0\|_{L^2}^{1-s_c}\|(-\Delta_\gamma)^\frac{1}{2} u_0\|_{L^2}^{s_c}
		> \|Q_{1,0}\|_{L^2}^{1-s_c}\|\nabla Q_{1,0}\|_{L^2}^{s_c}. \label{167}
\end{align}

In this paper, we investigate relations between them by focusing on the behavior of the solution to \eqref{NLS}.

\subsection{Main result}

First, we state the following result for the time behavior of solutions to \eqref{NLS}.

\begin{theorem}[Boundedness versus unboundedness \Rnum{1}]\label{Global versus blow-up or grow-up}
Let $d \geq 1$, $2_\ast < p+1 < 2^\ast$, $\gamma > 0$, and $0 < \mu < \min\{2,d\}$.
Let $j = 2, 3$.
\begin{itemize}
\item (Global well-posedness)
If $u_0 \in PW_{+,\,j}$, then a solution $u$ to \eqref{NLS} with the initial data $u_0$ satisfies $u(t) \in PW_{+,\,j}$ for each $t \in (T_\text{min},T_\text{max})$ and exists globally in both time directions.
In particular, $H^1$-norm of the solution $u$ is uniformly bounded in maximal lifespan.
\item (Brow-up or grow-up)
If $u_0 \in PW_{-,\,j}$, then a solution $u$ to \eqref{NLS} with a initial data $u_0$ satisfies $u(t) \in PW_{-,\,j}$ for each $t \in (T_\text{min},T_\text{max}) $and blows up or grows up in both time directions.
Moreover, if $u_0$ satisfies $u_0\in H_\text{rad}^1(\mathbb{R}^d)$ with $d \geq 2$ and $p \leq 5$ or $u_0\in |x|^{-1}L^2(\mathbb{R}^d)$ with $d \geq 1$, then $u$ blows up in both time directions.
\end{itemize}
\end{theorem}

Combining Theorem \ref{Dinh182} and Theorem \ref{Global versus blow-up or grow-up}, we obtain the following our main result.

\begin{theorem}[Equivalence of conditions on the initial data below the ground state]\label{Equivalence of L^2H^1 and K}
Let $d \geq 1$, $2_\ast < p+1 < 2^\ast$, $\gamma > 0$, and $0 < \mu < \min\{2,d\}$.
We assume that $u_0$ satisfies \eqref{165}.
The three conditions \eqref{166}, \eqref{169}, $K_\gamma(u_0) \geq 0$ are equivalent.
On the other hand, the three conditions \eqref{164}, \eqref{167}, $K_\gamma(u_0) < 0$ are equivalent.
In other words, $PW_{+,1} = PW_{+,2} = PW_{+,3}$ and $PW_{-,1} = PW_{-,2} = PW_{-,3}$ hold.
\end{theorem}

\begin{remark}
When $\gamma = 0$, it is well known that Theorem \ref{Equivalence of L^2H^1 and K} holds.
\end{remark}

\begin{corollary}\label{Negative energy}
Let $d \geq 1$, $2_\ast < p+1 < 2^\ast$, $\gamma > 0$, and $0 < \mu < \min\{2,d\}$.
If $u_0 \in H^1(\mathbb{R}^d)\setminus\{0\}$ satisfies $E_\gamma[u_0] \leq 0$, then $u_0 \in PW_{-,\,j}$ ($j = 1, 2, 3$).
\end{corollary}

Moreover, we state the result for time behavior with the radial initial data below the radial ground state $Q_{\omega,\gamma}$.
To state the result, we introduction existence of a radial ground state $Q_{\omega,\gamma}$.
For the equation including \eqref{SP} in \cite{HamIkeProceeding, HamIkeISAAC}:
\begin{align}
	-\omega Q_{\omega,V} + \Delta Q_{\omega,V} - VQ_{\omega,V}
		= -|Q_{\omega,V}|^{p-1}Q_{\omega,V}, \label{123}
\end{align}
we gave the following theorem.

\begin{proposition}[Existence of a radial ground state, \cite{HamIkeProceeding, HamIkeISAAC}]\label{Existence of a radial ground state}
Let $d \geq 3$, $2_\ast < p+1 < 2^\ast$, $x^\alpha \partial^\alpha V \in L^\frac{d}{2}(\mathbb{R}^d) + L^\infty(\mathbb{R}^d)$ for any $\alpha \in (\mathbb{Z}_{\geq 0})^d$ with $|\alpha| \leq 1$, $V \geq 0$, $x\cdot\nabla V \leq 0$, and $-2\omega_0 := \essinf_{x\in \mathbb{R}^d}(2V+x\cdot\nabla V) > -\infty$.
Let $V$ be radially symmetric.
Let $(\alpha,\beta)$ satisfy \eqref{104} and $\omega \geq \omega_0$.
Then, there exists a function $Q_{\omega,V} \in H_\text{rad}^1(\mathbb{R}^3)$ such that $Q_{\omega,V}$ attains $r_{\omega,V}^{\alpha,\beta}$, where
\begin{align*}
	r_{\omega,V}^{\alpha,\beta}
		:= \inf\left\{S_{\omega,V}(f):f\in H_\text{rad}^1(\mathbb{R}^d)\setminus\{0\},\ K_{\omega,V}^{\alpha,\beta}(f)=0\right\}.
\end{align*}
Moreover, if $x^\alpha \partial^\alpha V \in L^\frac{d}{2}(\mathbb{R}^d) + L^\infty(\mathbb{R}^d)$ for any $\alpha \in (\mathbb{Z}_{\geq 0})^d$ with $|\alpha| \leq 2$ and $3x\cdot\nabla V + x\nabla^2 Vx^T \leq 0$, then $\mathcal{M}_{\omega,V,\text{rad}}^{\alpha,\beta} = \mathcal{G}_{\omega,V,\text{rad}}$, where $\nabla^2 V$ denotes Hessian matrix of $V$ and
\begin{align*}
	\mathcal{M}_{\omega,V,\text{rad}}^{\alpha,\beta}
		&:= \left\{\phi\in H_\text{rad}^1(\mathbb{R}^d)\setminus\{0\}:S_{\omega,V}(\phi)=r_{\omega,V}^{\alpha,\beta},\ K_{\omega,V}^{\alpha,\beta}(\phi)=0\right\},\\
	\mathcal{G}_{\omega,V,\text{rad}}
		&:= \left\{\phi \in \mathcal{A}_{\omega,V,\text{rad}}:S_{\omega,V}(\phi)\leq S_{\omega,V}(\psi)\text{ for any }\psi\in \mathcal{A}_{\omega,V,\text{rad}}\right\},\\
	\mathcal{A}_{\omega,V,\text{rad}}
		&:= \left\{\psi\in H_\text{rad}^1(\mathbb{R}^d)\setminus\{0\}:S_{\omega,V}'(\psi)=0\right\}.
\end{align*}
\end{proposition}

\begin{remark}\label{Radial ground state in d geq 2}
If we replace $x^\alpha \partial^\alpha V \in L^\frac{d}{2}(\mathbb{R}^d) + L^\infty(\mathbb{R}^d)$ with $x^\alpha \partial^\alpha V \in L^\eta(\mathbb{R}^d) + L^\infty(\mathbb{R}^d)$ for some $\frac{d}{2} < \eta < \infty$, then Theorem \ref{Existence of a radial ground state} also holds in $d = 2$.
\end{remark}

\begin{remark}
Let $d \geq 2$, $2_\ast < p+1 < 2^\ast$, $\gamma > 0$, and $0 < \mu <2$.
From Proposition \ref{Existence of a radial ground state} and Remark \ref{Radial ground state in d geq 2}, $r_{\omega,\gamma}^{\alpha,\beta}$ is independent of $(\alpha,\beta)$.
When $d = 1$, $2_\ast < p+1 < 2^\ast$, $\gamma > 0$, and $0 < \mu <1$, $r_{\omega,\gamma}^{\alpha,\beta}$ is independent of $(\alpha,\beta)$ from Proposition \ref{r is independent of a and b}.
We express $r_{\omega,\gamma} := r_{\omega,\gamma}^{\alpha,\beta}$ for simplicity.
\end{remark}

\begin{remark}
By the definitions of $n_{\omega,\gamma}$ and $r_{\omega,\gamma}$, we have $r_{\omega,\gamma} \geq n_{\omega,\gamma}$.
In addition, if we assume $d \geq 2$, $2_\ast < p+1 < 2^\ast$, $\gamma > 0$, and $0 < \mu <2$, then
\begin{align*}
	r_{\omega,\gamma}
		> n_{\omega,\gamma}
		= n_{\omega,0}
\end{align*}
from Proposition \ref{Minimization problem}, Remark \ref{Remark Minimization problem}, Proposition \ref{Existence of a radial ground state}, and Remark \ref{Radial ground state in d geq 2}.
\end{remark}

Here, we state the result for time behavior with the radial initial data below the radial ground state $Q_{\omega,\gamma}$.

\begin{theorem}[Boundedness versus unboundedness \Rnum{2}]\label{Radial blow-up}
Let $d \geq 1$, $2_\ast < p+1 < 2^\ast$, $\gamma > 0$, and $0 < \mu < \min\{2,d\}$.
\begin{itemize}
\item (Global well-posedness)
If $u_0 \in PW_{+,4}$, then a solution $u$ to \eqref{NLS} satisfies $u(t) \in PW_{+,4}$ for each $t \in (T_\text{min},T_\text{max})$ and exists globally in both time directions, where
\begin{align*}
	PW_{+,4}
		:= \bigcup_{\omega>0}\left\{u_0 \in H_\text{rad}^1(\mathbb{R}^d):S_{\omega,\gamma}(u_0) < r_{\omega,\gamma}\text{ and }K_\gamma(u_0) \geq 0\right\}.
\end{align*}
In particular, $H^1$-norm of the solution $u$ is uniformly bounded in maximal lifespan.
\item (Blow-up or grow-up)
If $u_0 \in PW_{-,4}$, then a solution $u$ to \eqref{NLS} satisfies $u(t) \in PW_{-,4}$ for each $t \in (T_\text{min},T_\text{max})$ and blows up or grows up in both time directions, where
\begin{align*}
	PW_{-,4}
		:= \bigcup_{\omega>0}\left\{u_0 \in H_\text{rad}^1(\mathbb{R}^d):S_{\omega,\gamma}(u_0) < r_{\omega,\gamma}\text{ and }K_\gamma(u_0) < 0\right\}.
\end{align*}
Moreover, if $d \geq 2$ and $p \leq 5$, then $u$ blows up in both time directions.
\end{itemize}
\end{theorem}




\subsection{Organization of the paper}

The organization of the rest of this paper is as follows.
In section \ref{Preliminaries}, we define some notations and collect some tools.
In section \ref{Coercivity lemma and global well-posedness}, we prove coercivity lemma and get global well-posedness in Theorem \ref{Global versus blow-up or grow-up} and Theorem \ref{Radial blow-up} by using the coercivity lemmas.
In section \ref{Blow-up or grow-up}, we prove the blow-up or grow-up results in Theorem \ref{Global versus blow-up or grow-up} and Theorem \ref{Radial blow-up}.
In section \ref{Blow-up}, we prove the blow-up results in Theorem \ref{Global versus blow-up or grow-up} and Theorem \ref{Radial blow-up}.
In section \ref{Appendix A}, we show some properties of $n_{\omega,\gamma}$ in section \ref{Introduction}.
In section \ref{Appendix B}, we show some properties of $r_{\omega,\gamma}$ in section \ref{Introduction}.

\section{Preliminaries}\label{Preliminaries}

In this section, we define some notations and collect some tools. 

\subsection{Notations and definitions}

For nonnegative $X$ and $Y$, we write $X\lesssim Y$ to denote $X\leq CY$ for some $C>0$.
If $X\lesssim Y\lesssim X$ holds, we write $X\sim Y$.
For $1\leq p\leq\infty$, $L^p=L^p(\mathbb{R}^d)$ denotes the usual Lebesgue space.
$H^1(\mathbb{R}^d)$ denotes the usual Sobolev space.
We note that $H^1(\mathbb{R}^d)$ is a real Hilbert space with an inner product:
\begin{align*}
	\langle f, g \rangle_{H^1}
		= \langle f, g\rangle_{L^2} + \langle \nabla f, \nabla g\rangle_{L^2}
		:= \text{Re}\int_{\mathbb{R}^d} (f(x)\overline{g(x)} + \nabla f(x)\cdot \overline{\nabla g(x)})dx.
\end{align*}

\subsection{Some tools}
In this section, we collect some tools used in this paper.\\

The following generalized Hardy's inequality assures that the energy $E_\gamma$ is well-defined on $H^1(\mathbb{R}^d)$.

\begin{lemma}[Generalized Hardy's inequality, \cite{ZhaZhe14}]
Let $1 < q < \infty$ and $0 < \mu < d$.
Then, the following inequality holds:
\begin{align*}
\int_{\mathbb{R}^d}\frac{1}{|x|^\mu}|f(x)|^qdx
	\lesssim_{q,\mu}\||\nabla|^\frac{\mu}{q}f\|_{L^q}.
\end{align*}
\end{lemma}

\begin{proposition}[Pohozaev identities without a potential, \cite{Caz03}]\label{Pohozaev identity}
Let $d \geq 1$ and $2 < p+1 < 2^\ast$.
The ground state $Q_{1,0}$ for the elliptic equation \eqref{SP} with $\omega=1$ and $\gamma=0$ satisfies the following Pohozaev identities:
\begin{align*}
	\|Q_{1,0}\|_{L^{p+1}}^{p+1}
		= \frac{2(p+1)}{d+2-(d-2)p}\|Q_{1,0}\|_{L^2}^2,\ \ \ 
	\|Q_{1,0}\|_{L^{p+1}}^{p+1}
		= \frac{2(p+1)}{d(p-1)}\|\nabla Q_{1,0}\|_{L^2}^2.
\end{align*}
\end{proposition}

For the proof of this proposition, see \cite[Lemma 8.1.2]{Caz03}.\\

Using Proposition \ref{Pohozaev identity}, we have
\begin{align}
	E_0[Q_{1,0}]
		= \frac{dp-(d+4)}{2d(p-1)}\|\nabla Q_{1,0}\|_{L^2}^2\ \ \text{ and }\ \ 
	C_\text{GN}
		= \frac{2(p+1)}{d(p-1)}\frac{1}{\|Q_{1,0}\|_{L^2}^\frac{d+2-(d-2)p}{2}\|\nabla Q_{1,0}\|_{L^2}^\frac{dp-(d+4)}{2}}. \label{111}
\end{align}

\begin{proposition}[Gagliardo-Nirenberg inequality with the inverse potential]\label{Gagliardo-Nirenberg inequality with inverse potential}
Let $d \geq 1$, $2 < p+1 < 2^\ast$, $\gamma > 0$, and $0 < \mu < \min\{2,d\}$.
Then, the following inequality holds:
\begin{align*}
	\|f\|_{L^{p+1}}^{p+1}
		< C_\text{GN}\|f\|_{L^2}^{p+1-\frac{d(p-1)}{2}}\|(-\Delta_\gamma)^\frac{1}{2} f\|_{L^2}^\frac{d(p-1)}{2}
\end{align*}
for any $f \in H^1(\mathbb{R}^d)\setminus\{0\}$, where $C_\text{GN}$ is the best constant and is defined in Proposition \ref{Gagliardo-Nirenberg inequality without a potential}.
\end{proposition}

\begin{proof}
The inequality holds by Proposition \ref{Gagliardo-Nirenberg inequality without a potential} and $\gamma>0$.
We set the best constant $C_\text{GN}^\dagger$ and prove $C_\text{GN}^\dagger = C_\text{GN}$.
We define a functional
\begin{align*}
	J_{\gamma}(f)
		:=\frac{\|f\|_{L^{p+1}}^{p+1}}{\|f\|_{L^2}^{p+1-\frac{d(p-1)}{2}}\|(-\Delta_\gamma)^\frac{1}{2} f\|_{L^2}^\frac{d(p-1)}{2}}
\end{align*}
for $f \in H^1(\mathbb{R}^d)\setminus\{0\}$.
Proposition \ref{Gagliardo-Nirenberg inequality without a potential} and $\gamma > 0$ imply $C_\text{GN} = J_0(Q_{1,0}) \geq J_0(f) \geq J_\gamma(f)$ for any $f\in H^1(\mathbb{R}^d)\setminus\{0\}$.
This inequality deduces $C_\text{GN} \geq C_\text{GN}^\dagger$.
On the other hand, we consider a sequence $\{Q_{1,0}(n\,\cdot\,)\}$.
Then, we have $J_\gamma(Q_{1,0}(n\,\cdot\,))\leq C_\text{GN}^\dagger$ for each $n\in\mathbb{N}$.
Thus, it follows that
\begin{align*}
	C_\text{GN}^\dagger
		&\geq \lim_{n\rightarrow\infty}J_\gamma(Q_{1,0}(n\,\cdot\,))
		= J_0(Q_{1,0})
		= C_\text{GN}.
\end{align*}
Therefore, we obtain $C_\text{GN}^\dagger=C_\text{GN}$.
\end{proof}

\begin{lemma}[Radial Sobolev inequality, \cite{OgaTsu91}]\label{Radial Sobolev inequality}
Let $1 \leq p$.
For a radial function $f\in H^1(\mathbb{R}^d)$, it follows that
\begin{align*}
	\|f\|_{L^{p+1}(R\leq|x|)}^{p+1}
		\lesssim \frac{1}{R^\frac{(d-1)(p-1)}{2}}\|f\|_{L^2(R\leq|x|)}^\frac{p+3}{2}\|\nabla f\|_{L^2(R\leq|x|)}^\frac{p-1}{2}
\end{align*}
for any $R>0$, where the implicit constant is independent of $R$ and $f$.
\end{lemma}

\begin{proposition}[Localized virial identity, \cite{TaoVisZha07}, \cite{Din182}]\label{Virial identity}
Given a suitable real-valued weight function $w\in C^\infty(\mathbb{R}^d)$ and the solution $u(t)$ to \eqref{NLS}, we define
\begin{align*}
	I(t)
		:=\int_{\mathbb{R}^d}w(x)|u(t,x)|^2dx.
\end{align*}
Then, it follows that
\begin{align*}
	I'(t)
		=2\text{Im}\int_{\mathbb{R}^d}\overline{u}\nabla u\cdot\nabla wdx, \label{112}
\end{align*}
If $w$ is radial, then we have
\begin{align*}
	I'(t)
		= 2\text{Im}\int_{\mathbb{R}^d}\frac{x\cdot\nabla u}{r}\overline{u}w'dx,
\end{align*}
\begin{align*}
	&I''(t)
		= \int_{\mathbb{R}^d}F_1|x\cdot\nabla u|^2dx+4\int_{\mathbb{R}^d}\frac{w'}{r}|\nabla u|^2dx-\int_{\mathbb{R}^d}F_2|u|^{p+1}dx \notag \\
	&\hspace{6.0cm}-\int_{\mathbb{R}^d}F_3|u|^2dx+2\mu\int_{\mathbb{R}^d}w'\frac{\gamma}{r^{\mu+1}}|u|^2dx.
\end{align*}
where
\begin{align*}
	F_1(w,r)
		:= 4\left(\frac{w''}{r^2}-\frac{w'}{r^3}\right),\ \ \ 
	F_2(w,r)
		:= \frac{2(p-1)}{p+1}\left(w''+\frac{d-1}{r}w'\right)
\end{align*}
\begin{align*}
	F_3(w,r)
		:= w^{(4)}+\frac{2(d-1)}{r}w^{(3)}+\frac{(d-1)(d-3)}{r^2}w''+\frac{(d-1)(3-d)}{r^3}w'.
\end{align*}
\end{proposition}

To prove that $r_{\omega,\gamma}^{\alpha,\beta}$ is independent of $(\alpha,\beta)$, we prepare the following lemma.

\begin{lemma}[Positivity of $K_{\omega,\gamma}^{\alpha,\beta}$ near the origin]\label{Positivity of K}
Let $d \geq 1$, $2_\ast < p+1 < 2^\ast$, $\gamma > 0$, and $0 < \mu < \min\{2,d\}$.
Let $(\alpha,\beta)$ satisfy \eqref{104}.
Suppose that $\{f_n\}$ is a bounded sequence in $H^1(\mathbb{R}^d)\setminus\{0\}$ and satisfies $\|\nabla f_n\|_{L^2}\longrightarrow0$ as $n\rightarrow\infty$.
Then, there exists $n_0\in \mathbb{N}$ such that
\begin{align*}
	K_{\omega,\gamma}^{\alpha,\beta}(f_n)
		> 0
\end{align*}
for any $n\geq n_0$.
\end{lemma}

\begin{proof}
We take a positive constant $C$ with $\displaystyle \sup_{n\in\mathbb{N}}\|f_n\|_{L^2}\leq C$.
Applying the Gagliardo--Nirenberg inequality (Proposition \ref{Gagliardo-Nirenberg inequality without a potential}), we have
\begin{align*}
	K_{\omega,\gamma}^{\alpha,\beta}(f_n)
		&\geq \left(\frac{2\alpha-(d-2)\beta}{2}-\frac{(p+1)\alpha-d\beta}{p+1}C_\text{GN}\,C^{p+1-\frac{d(p-1)}{2}}\|\nabla f_n\|_{L^2}^{\frac{d(p-1)}{2}-2}\right)\|\nabla f_n\|_{L^2}^2.
\end{align*}
When $\|\nabla f_n\|_{L^2}\neq0$ is sufficiently small, we obtain $K_{\omega,\gamma}^{\alpha,\beta}(f_n)>0$.
\end{proof}

We prove that $r_{\omega,\gamma}^{\alpha,\beta}$ is independent of $(\alpha,\beta)$.

\begin{proposition}\label{r is independent of a and b}
Let $d = 1$, $2_\ast < p+1 < 2^\ast$, $\gamma > 0$, and $0 < \mu < 1$.
Let $(\alpha,\beta)$ satisfies \eqref{104}.
Then, we have
\begin{align*}
	r_{\omega,\gamma}^{\alpha,\beta}
		= \inf_{c \, \in \, \mathcal{C}}\max_{\tau \in [0,1]}S_{\omega,\gamma}(c(\tau)),
\end{align*}
where
\begin{align*}
	\mathcal{C}
		:= \left\{c\in C([0,1];H_\text{rad}^1(\mathbb{R})):c(0)=0,\ S_{\omega,\gamma}(c(1))<0\right\}.
\end{align*}
In particular, $r_{\omega,\gamma}^{\alpha,\beta}$ is independent of $(\alpha,\beta)$.
\end{proposition}

The proof is based on \cite[Lemma 2.3]{IkeInu17}.

\begin{proof}
We set
\begin{align*}
	\mathscr{R}
		:= \inf_{c \, \in \, \mathcal{C}}\max_{\tau \in [0,1]}S_{\omega,\gamma}(c(\tau)).
\end{align*}
To prove $\mathscr{R} \leq r_{\omega,\gamma}^{\alpha,\beta}$, we prove that there exists $\{c_n\} \subset \mathcal{C}$ such that
\begin{align*}
	\max_{\tau \in [0,1]}S_{\omega,\gamma}(c_n(\tau))
		\longrightarrow r_{\omega,\gamma}^{\alpha,\beta}
\end{align*}
as $n \rightarrow \infty$.
We take a minimizing sequence $\{\varphi_n\}$ to $r_{\omega,\gamma}^{\alpha,\beta}$, that is,
\begin{align*}
	S_{\omega,\gamma}(\varphi_n)
		\longrightarrow r_{\omega,\gamma}^{\alpha,\beta}\ \ \text{ as }\ \ n \rightarrow \infty\ \ \text{ and }\ \ 
	K_{\omega,\gamma}^{\alpha,\beta}(\varphi_n)
		= 0\ \ \text{ for each }\ \ n \in \mathbb{N}.
\end{align*}
We set $\widetilde{c}_n(\tau) := e^{\alpha\tau}\varphi_n(e^{\beta\tau}\,\cdot\,)$ for $\tau \in \mathbb{R}$.
Then,
\begin{align*}
	S_{\omega,\gamma}(\widetilde{c}_n(\tau))
		& = \frac{\omega}{2}e^{(2\alpha-\beta)\tau}\|\varphi_n\|_{L^2}^2 + \frac{1}{2}e^{(2\alpha+\beta)\tau}\|\nabla \varphi_n\|_{L^2}^2\\
		& + \frac{1}{2}e^{\{2\alpha-(1-\mu)\beta\}\tau}\int_{\mathbb{R}}\frac{\gamma}{|x|^\mu}|\varphi_n(x)|^2dx - \frac{1}{p+1}e^{\{(p+1)\alpha-\beta\}\tau}\|\varphi_n\|_{L^{p+1}}^{p+1},
\end{align*}
so $S_{\omega,\gamma}(\widetilde{c}_n(\tau)) < 0$ for sufficiently large $\tau > 0$.
Moreover, we have $\max_{\tau \in \mathbb{R}}S_{\omega,\gamma}(\widetilde{c}_n(\tau)) = S_{\omega,\gamma}(\widetilde{c}_n(0)) = S_{\omega,\gamma}(\varphi_n) \longrightarrow r_{\omega,\gamma}^{\alpha,\beta}$ as $n \rightarrow \infty$ by $K_{\omega,\gamma}^{\alpha,\beta}(\varphi_n) = 0$.
We define a function $c_n'$ for $\tau \in [-L,L]$ as follows:
\begin{equation*}
c_n'(\tau):=
\begin{cases}
\hspace{-0.4cm}&\displaystyle{\widetilde{c}_n(\tau),\hspace{3.9cm}\Bigl(-\frac{L}{2}\leq \tau \leq L\Bigr),}\\
\hspace{-0.4cm}&{\left\{\frac{2}{L}(\tau+L)\right\}^M\widetilde{c}_n\left(-\frac{L}{2}\right),\hspace{0.96cm}\Bigl(-L\leq \tau<-\frac{L}{2}\Bigr)}.
\end{cases}
\end{equation*}
Then, $c_n' \in C ([-L,L];H_\text{rad}^1(\mathbb{R}))$, $S_{\omega,\gamma}(c_n'(L)) < 0$, and $\max_{\tau \in [-L,L]}S_{\omega,\gamma}(c_n'(\tau)) = S_{\omega,\gamma}(\varphi_n) \longrightarrow r_{\omega,\gamma}^{\alpha,\beta}$ when $L > 0$ and $M = M(n)$ are sufficiently large.
Changing variables, we consider a sequence $\{c_n'(2L\tau-L)\}$.
Then, $c_n'(2L\tau-L) \in \mathcal{C}$ satisfies $\max_{\tau\in[0,1]}S_{\omega,\gamma}(c_n'(2L\tau-L)) \longrightarrow r_{\omega,\gamma}^{\alpha,\beta}$ as $n \rightarrow \infty$, which implies $\mathscr{R} \leq r_{\omega,\gamma}^{\alpha,\beta}$.
To prove $\mathscr{R} \geq r_{\omega,\gamma}^{\alpha,\beta}$, we prove
\begin{align*}
	c([0,1])
		\cap \left\{\varphi\in H_\text{rad}^1(\mathbb{R})\setminus\{0\}:K_{\omega,\gamma}^{\alpha,\beta}(\varphi)=0\right\}
		\neq \emptyset
\end{align*}
for any $c \in \mathcal{C}$.
We take any $c \in \mathcal{C}$, that is, $c(0) = 0$ and $S_{\omega,\gamma}(c(1)) < 0$.
Then, we have
\begin{align*}
	K_{\omega,\gamma}^{\alpha,\beta}(c(1))
		\leq \{(p+1)\alpha-\beta\}S_{\omega,\gamma}(c(1))
		<0.
\end{align*}
From Lemma \ref{Positivity of K}, it follows that $K_{\omega,\gamma}^{\alpha,\beta}(c(\tau)) > 0$ for some $\tau \in (0,1)$.
By the continuity, there exists $\tau_0 \in (0,1)$ such that $K_{\omega,\gamma}^{\alpha,\beta}(c(\tau_0)) = 0$.
Therefore, we obtain $\mathscr{R} = r_{\omega,\gamma}^{\alpha,\beta}$.
\end{proof}

\section{Coercivity lemma and global well-posedness}\label{Coercivity lemma and global well-posedness}

In this section, we prove coercivity lemmas.
Then, we prove global well-posedness in Theorem \ref{Global versus blow-up or grow-up} and Theorem \ref{Radial blow-up}.

\begin{lemma}[Coercivity \Rnum{1}]\label{Coercivity 1}
Let $d \geq 1$, $2_\ast < p+1 < 2^\ast$, $\gamma > 0$, and $0 < \mu < \min\{2,d\}$.
Let $u_0 \in PW_{+,3} \cup PW_{-,3}$.
We take a positive constant $\delta > 0$ satisfying
\begin{align*}
	M[u_0]^{1-s_c}E_\gamma[u_0]^{s_c}
		< (1-\delta)M[Q_{1,0}]^{1-s_c}E_\gamma[Q_{1,0}]^{s_c}
\end{align*}
\begin{itemize}
\item ($PW_{+,3}$ case)
If $u_0 \in PW_{+,3}$, then a solution $u$ to \eqref{NLS} with a initial data $u_0$ satisfies the following: there exists $\delta' > 0$ such that
\begin{align*}
	\|u(t)\|_{L_x^2}^{1-s_c}\|(-\Delta_\gamma)^\frac{1}{2} u(t)\|_{L_x^2}^{s_c}
		< (1-\delta')\|Q_{1,0}\|_{L_x^2}^{1-s_c}\|\nabla Q_{1,0}\|_{L_x^2}^{s_c}
\end{align*}
for any $t\in (T_\text{min},T_\text{max})$.
\item ($PW_{-,3}$ case)
If $u_0 \in PW_{-,3}$, then a solution $u$ to \eqref{NLS} with a initial data $u_0$ satisfies the following: there exists $\delta' > 0$ such that
\begin{align*}
	\|u(t)\|_{L_x^2}^{1-s_c}\|(-\Delta_\gamma)^\frac{1}{2} u(t)\|_{L_x^2}^{s_c}
		> (1+\delta')\|Q_{1,0}\|_{L_x^2}^{1-s_c}\|\nabla Q_{1,0}\|_{L_x^2}^{s_c}
\end{align*}
for any $t\in (T_\text{min},T_\text{max})$.
\end{itemize}
\end{lemma}

\begin{proof}
Using Proposition \ref{Gagliardo-Nirenberg inequality with inverse potential}, we have
\begin{align*}
	&(1-\delta)^\frac{1}{s_c}M[Q_{1,0}]^\frac{1-s_c}{s_c}E_0[Q_{1,0}]
		> M[u_0]^\frac{1-s_c}{s_c}E_\gamma[u_0]\\
		&\hspace{2.0cm} \geq \|u(t)\|_{L_x^2}^\frac{2(1-s_c)}{s_c}\left(\frac{1}{2}\|(-\Delta_\gamma)^\frac{1}{2} u(t)\|_{L_x^2}^2-\frac{1}{p+1}C_\text{GN}\|u(t)\|_{L_x^2}^{p+1-\frac{d(p-1)}{2}}\|(-\Delta_\gamma)^\frac{1}{2} u(t)\|_{L_x^2}^\frac{d(p-1)}{2}\right)\\
		&\hspace{2.0cm} = \frac{1}{2}\|u(t)\|_{L_x^2}^\frac{2(1-s_c)}{s_c}\|(-\Delta_\gamma)^\frac{1}{2} u(t)\|_{L_x^2}^2-\frac{2}{d(p-1)}\cdot\frac{\|u(t)\|_{L_x^2}^\frac{d(p-1)(1-s_c)}{2s_c}\|(-\Delta_\gamma)^\frac{1}{2} u(t)\|_{L_x^2}^\frac{d(p-1)}{2}}{\|Q_{1,0}\|_{L_x^2}^\frac{d+2-(d-2)p}{2}\|\nabla Q_{1,0}\|_{L_x^2}^\frac{dp-(d+4)}{2}}.
\end{align*}
This inequality implies
\begin{align*}
	(1-\delta)^\frac{1}{s_c}
		> g\left(\frac{\|u(t)\|_{L^2}^\frac{1-s_c}{s_c}\|(-\Delta_\gamma)^\frac{1}{2} u(t)\|_{L^2}}{\|Q_{1,0}\|_{L^2}^\frac{1-s_c}{s_c}\|\nabla Q_{1,0}\|_{L^2}}\right),
\end{align*}
where a function $g$ is defined as $g(y) = \frac{d(p-1)}{dp-(d+4)}y^2-\frac{4}{dp-(d+4)}y^\frac{d(p-1)}{2}$ for $y\geq 0$.
Then,
$g$ has a local minimum at $y_0=0$ and a local maximum at $y_1=1$.
Combining these facts and the assumption of Lemma \ref{Coercivity 1}, we obtain the desired result.
\end{proof}

The $PW_{+,3}$ case result in  Lemma \ref{Coercivity 1} deduced global well-posedness in Theorem \ref{Global versus blow-up or grow-up} with $j = 3$.

\begin{proof}[Proof of global well-posedness in Theorem \ref{Global versus blow-up or grow-up} with $j = 3$]
The desired result follows from the fact that $H^1$-norm of the solutions is uniformly bounded with respect to time $t$.
\end{proof}

For simplicity, ``$(PW_+, m_{\omega,\gamma})$ denotes $(PW_{+,2}, n_{\omega,\gamma})$ or $(PW_{+,4},r_{\omega,\gamma})$'' and ``$(PW_-,m_{\omega,\gamma})$ denotes $(PW_{-,2},n_{\omega,\gamma})$ or $(PW_{-,4},r_{\omega,\gamma})$''.

\begin{lemma}[Coercivity \Rnum{2}]\label{Coercivity 2}
Let $d \geq 1$, $2_\ast < p+1 < 2^\ast$, $\gamma > 0$, and $0 < \mu < \min\{2,d\}$.
\begin{itemize}
\item ($PW_{+}$ case)
If $u_0 \in PW_{+}$, then a solution $u$ to \eqref{NLS} with the initial data $u_0$ satisfies $u(t) \in PW_{+}$ for each $t \in (T_\text{min},T_\text{max})$.
\item ($PW_{-}$ case)
If $u_0 \in PW_{-}$, then a solution $u$ to \eqref{NLS} with the initial data $u_0$ satisfies $u(t) \in PW_{-}$ for each $t \in (T_\text{min},T_\text{max})$ and
\begin{align*}
	K_\gamma(u(t))
		< 4(S_{\omega,\gamma}(u_0)-m_{\omega,\gamma})
		<0.
\end{align*}
\end{itemize}
\end{lemma}

\begin{proof}
When $K_\gamma(u_0) = 0$, we have $u_0 \equiv 0$ by the definition of $m_{\omega,\gamma}$.
Thus, Lemma \ref{Coercivity 2} holds.
Suppose that $K_\gamma(u_0) \neq 0$.
If there exists $t_0 \in (T_\text{min},T_\text{max})$ such that	$K_\gamma(u(t_0)) = 0$, then
\begin{align*}
	S_{\omega,\gamma}(u(t_0))
		= S_{\omega,\gamma}(u_0)
		< m_{\omega,\gamma}
		\leq S_{\omega,\gamma}(u(t_0)).
\end{align*}
This is contradiction.
Therefore, $K_\gamma(u(t)) \neq 0$ for each $t \in (T_\text{min}, T_\text{max})$.
In particular, the sign of $K_\gamma(u(t))$ corresponds with that of $K_\gamma(u_0)$ by the continuity of the solution.
Let $K_\gamma(u_0) < 0$.
We define a function
\begin{align*}
	J_{\omega,\gamma}(\lambda)
		& = S_{\omega,\gamma}(e^{d\lambda}u(e^{2\lambda}\,\cdot\,)).
\end{align*}
We note that
\begin{align*}
	J_{\omega,\gamma}(0)
		= S_{\omega,\gamma}(u),\ \ \ 
	\frac{d}{d\lambda}J_{\omega,\gamma}(0)
		= K_\gamma(u),\ \ \ 
	\frac{d^2}{d\lambda^2}J_{\omega,\gamma}(\lambda)
		< 4\frac{d}{d\lambda}J_{\omega,\gamma}(\lambda).
\end{align*}
The equation $\frac{d}{d\lambda}J_{\omega,\gamma}(\lambda) = 0$ for $\lambda$ has only one negative solution.
We set that the solution $\lambda = \lambda_0 < 0$.
Integrating $\frac{d^2}{d\lambda^2}J_{\omega,\gamma}(\lambda)
< 4\frac{d}{d\lambda}J_{\omega,\gamma}(\lambda)$ over $[\lambda_0,0]$, we have
\begin{align*}
	K_\gamma(u) - 0
		< 4(S_{\omega,\gamma}(u)-J_{\omega,\gamma}(\lambda_0))
		\leq 4(S_{\omega,\gamma}(u)-m_{\omega,\gamma})
		< 0.
\end{align*}
Therefore, we obtain
\begin{align*}
	K_\gamma(u(t))
		< 4(S_{\omega,\gamma}(u_0)-m_{\omega,\gamma})
		< 0
\end{align*}
for any $t\in (T_\text{min}, T_\text{max})$.
\end{proof}

As a corollary, global well-posedness in Theorem \ref{Global versus blow-up or grow-up} and Theorem \ref{Radial blow-up} holds.

\begin{corollary}[Global well-posedness]\label{Global well-posedness}
Let $d \geq 1$, $2_\ast < p+1 < 2^\ast$, $\gamma > 0$, and $0 < \mu < \min\{2,d\}$.
If $u_0 \in PW_{+}$, then a solution $u$ to \eqref{NLS} with the initial data $u_0$ exists globally in time.
\end{corollary}

\begin{proof}
From $u_0 \in PW_{+}$ and Lemma \ref{Coercivity 2}, we have $u(t) \in PW_{+}$ for each $t \in (T_\text{min},T_\text{max})$.
$K_\gamma(u(t)) \geq 0$ deduces
\begin{align*}
	2\|(-\Delta_\gamma)^\frac{1}{2}u(t)\|_{L^2}^2
		\geq \frac{d(p-1)}{p+1}\|u(t)\|_{L^{p+1}}^{p+1}.
\end{align*}
Therefore, we obtain
\begin{align*}
	m_{\omega,\gamma}
		> S_{\omega,\gamma}(u_0)
		\geq \frac{\omega}{2}\|u(t)\|_{L^2}^2 + \frac{d(p-1)-4}{2d(p-1)}\|(-\Delta_\gamma)^\frac{1}{2}u(t)\|_{L^2}^2
		\gtrsim \|u(t)\|_{H^1}^2,
\end{align*}
which implies the desired result.
\end{proof}

\section{Blow-up or grow-up}\label{Blow-up or grow-up}

In this section, we prove blow-up or grow-up results in Theorem \ref{Global versus blow-up or grow-up} and Theorem \ref{Radial blow-up}.
We consider only positive time direction since we get the same conclusion for negative time direction by taking the complex conjugate of the equation and replacing $t$ with $-t$.
The proof is based on \cite{DuWuZha16}.\\

Before we prove the results, we define the following functions for each $R > 0$.
A cut-off function $\mathscr{X}_R\in C_0^\infty(\mathbb{R}^d)$ is radially symmetric and satisfies
\begin{equation}
\mathscr{X}_R(r):=R^2\mathscr{X}\left(\frac{r}{R}\right),\ \text{ where }\ \mathscr{X}(r):= \label{002}
\begin{cases}
\hspace{-0.4cm}&\displaystyle{\ \ \ \,r^2\ \ \ \ \ \,(0\leq r\leq1),}\\
\hspace{-0.4cm}&\displaystyle{smooth\ \ (1\leq r\leq 3),}\\
\hspace{-0.4cm}&\displaystyle{\ \ \ \ 0\ \ \ \ \ \ \,(3\leq r),}
\end{cases}
\end{equation}
$\mathscr{X}''(r)\leq2$ $(r\geq0)$, and $r=|x|$.
A cut-off function $\mathscr{Y}_R\in C_0^\infty(\mathbb{R}^d)$ is radially symmetric and satisfies
\begin{equation}
\mathscr{Y}_R(r):=\mathscr{Y}\left(\frac{r}{R}\right),\ \text{ where }\ \mathscr{Y}(r):= \label{003}
\begin{cases}
\hspace{-0.4cm}&\displaystyle{\ \ \ \ 0\ \ \ \ \ \ \,(0\leq r\leq1/2),}\\
\hspace{-0.4cm}&\displaystyle{smooth\ \ (1/2\leq r\leq 1),}\\
\hspace{-0.4cm}&\displaystyle{\ \ \ \ 1\ \ \ \ \ \ \,(1\leq r),}
\end{cases}
\end{equation}
and $0 \leq \mathscr{Y}'(r) \leq 3\ \ (r \geq 0)$.

\begin{lemma}\label{Lemma1 for blows-up or grows-up}
Let $d\geq 1$, $2 \leq p+1 < 2^\ast$, $\gamma > 0$, and $0 < \mu < \min\{2,d\}$.
We assume that $u\in C([0,\infty);H^1)$ be a solution to \eqref{NLS} satisfying
\begin{align*}
	C_0
		:= \sup_{t\in[0,\infty)}\|\nabla u\|_{L_x^2}
		< \infty.
\end{align*}
Then, it follows that
\begin{align*}
	\int_{|x|>R}|u(t,x)|^2dx
		\leq o_R(1) + \eta
\end{align*}
for any $\eta>0$, $R>0$, and $t \in \Bigl[0,\frac{\eta R}{6C_0\|u\|_{L_x^2}}\Bigr]$, where $o_R(1)$ goes to zero as $R \rightarrow \infty$ and is independent of $t$.
\end{lemma}

\begin{proof}
We define a function
\begin{align*}
	I(t)
		:=\int_{\mathbb{R}^d}\mathscr{Y}_R(x)|u(t,x)|^2dx,
\end{align*}
where $\mathscr{Y}_R$ is defined as \eqref{003}.
Using Proposition \ref{Virial identity},
\begin{align*}
	I(t)
		&= I(0) + \int_0^t\frac{d}{ds}I(s)ds
		\leq I(0) + \int_0^t|I'(s)|ds\\
		&\leq I(0) + 2t\|\nabla\mathscr{Y}_R\|_{L_x^\infty}\sup_{t\in[0,\infty)}\|\nabla u(t)\|_{L_x^2}\|u\|_{L_x^2}
		\leq I(0) + \frac{6C_0\|u\|_{L_x^2}t}{R}
\end{align*}
for any $t \in [0,\infty)$.
Since $u_0\in H^1(\mathbb{R}^d)$,
\begin{align*}
	I(0)
		= \int_{\mathbb{R}^d}\mathscr{Y}_R(x)|u_0(x)|^2dx
		\leq \int_{|x|>\frac{R}{2}}|u_0(x)|^2dx
		= o_R(1).
\end{align*}
Therefore, we have
\begin{align*}
	\int_{|x|>R}|u(t,x)|^2dx
		\leq o_R(1) + \eta.
\end{align*}
\end{proof}

\begin{lemma}\label{Lemma2 for blows-up or grows-up}
Let $d \geq 1$, $2 < p+1 < 2^\ast$, $\gamma > 0$, and $0 < \mu < \min\{2,d\}$.
Let $u \in C([0,\infty);H^1(\mathbb{R}^d))$ be a solution to \eqref{NLS}.
We define a function
\begin{align*}
	I(t)
		:= \int_{\mathbb{R}^d}\mathscr{X}_R(x)|u(t,x)|^2dx,
\end{align*}
where $\mathscr{X}_R$ is defined as \eqref{002}.
Then, for $q \in (p+1,2^\ast)$, there exist constants $C = C(q,\|u_0\|_{L_x^2},C_0) > 0$ and $\theta_q > 0$ such that the estimate
\begin{align*}
	&I''(t)
		\leq 4K_\gamma(u(t)) + C\,\|u(t)\|_{L_x^2(R\leq|x|)}^{(p+1)\theta_q} + \frac{C}{R^2}
\end{align*}
holds for any $R > 0$ and $t \in [0,\infty)$, where $\theta_q := \frac{2\{q-(p+1)\}}{(p+1)(q-2)}\in(0,\frac{2}{p+1})$ and $C_0$ is given in Lemma \ref{Lemma1 for blows-up or grows-up}.
\end{lemma}

\begin{proof}
Using Proposition \ref{Virial identity}, we have
\begin{align*}
	I''(t)
		= 4K_\gamma(u(t)) + \mathcal{R}_1 + \mathcal{R}_2 + \mathcal{R}_3 + \mathcal{R}_4,
\end{align*}
where $\mathcal{R}_k=\mathcal{R}_k(t)$ $(k=1,2,3,4)$ are defined as
\begin{align}
	\mathcal{R}_1
		&:= 4\int_{\mathbb{R}^d}\left\{\frac{1}{r^2}\mathscr{X}''\left(\frac{r}{R}\right)-\frac{R}{r^3}\mathscr{X}'\left(\frac{r}{R}\right)\right\}|x\cdot\nabla u|^2dx+4\int_{\mathbb{R}^d}\left\{\frac{R}{r}\mathscr{X}'\left(\frac{r}{R}\right)-2\right\}|\nabla u(t,x)|^2dx, \label{160} \\
	\mathcal{R}_2
		&:= -\frac{2(p-1)}{p+1}\int_{\mathbb{R}^d}\left\{\mathscr{X}''\left(\frac{r}{R}\right)+\frac{(d-1)R}{r}\mathscr{X}'\left(\frac{r}{R}\right)-2d \right\}|u(t,x)|^{p+1}dx, \label{161} \\
	\mathcal{R}_3
		&:= -\int_{\mathbb{R}^d}\left\{\frac{1}{R^2}\mathscr{X}^{(4)}\left(\frac{r}{R}\right)+\frac{2(d-1)}{Rr}\mathscr{X}^{(3)}\left(\frac{r}{R}\right)+\frac{(d-1)(d-3)}{r^2}\mathscr{X}''\left(\frac{r}{R}\right)\right. \notag \\
		&\hspace{7.0cm} \left.+\frac{(d-1)(3-d)R}{r^3}\mathscr{X}'\left(\frac{r}{R}\right)\right\}|u(t,x)|^2dx, \label{162} \\
	\mathcal{R}_4
		&:= 2\mu\int_{R\leq|x|}\left\{\frac{R}{r}\mathscr{X}'\left(\frac{r}{R}\right)-2\right\}\frac{\gamma}{|x|^\mu}|u(t,x)|^2dx. \label{163}
\end{align}
We set
\begin{align*}
	\Omega
		:= \left\{x \in \mathbb{R}^d : \frac{1}{r^2}\mathscr{X}''\left(\frac{r}{R}\right)-\frac{R}{r^3}\mathscr{X}'\left(\frac{r}{R}\right) \leq 0 \right\}.
\end{align*}
By $\mathscr{X}'(\frac{r}{R}) \leq \frac{2r}{R}$, we have
\begin{align*}
	\mathcal{R}_1
		& \leq 4\int_{\Omega^c}\left\{\mathscr{X}''\left(\frac{r}{R}\right)-\frac{R}{r}\mathscr{X}'\left(\frac{r}{R}\right)\right\}|\nabla u(t,x)|^2dx + 4\int_{\Omega^c}\left\{\frac{R}{r}\mathscr{X}'\left(\frac{r}{R}\right)-2\right\}|\nabla u(t,x)|^2dx\\
		& = 4\int_{\Omega^c}\left\{\mathscr{X}''\left(\frac{r}{R}\right)-2\right\}|\nabla u(t,x)|^2dx
		\leq 0,
\end{align*}
where $\Omega^c$ denotes a complement of $\Omega$.\\
Next, we estimate $\mathcal{R}_2$.
Applying H\"older's inequality and Sobolev's embedding, we have
\begin{align*}
	\mathcal{R}_2
		& = -\frac{2(p-1)}{p+1}\int_{R\leq|x|\leq3R}\left\{\mathscr{X}''\left(\frac{r}{R}\right) + \frac{(d-1)R}{r}\mathscr{X}'\left(\frac{r}{R}\right)\right\}|u(t,x)|^{p+1}dx + \frac{4d(p-1)}{p+1}\int_{R\leq|x|}|u(t,x)|^{p+1}dx\\
		& \leq C\|u(t)\|_{L_x^{p+1}(R\leq|x|)}^{p+1}
		\leq C\|u(t)\|_{L_x^q(R\leq|x|)}^{(p+1)(1-\theta_q)}\|u(t)\|_{L_x^2(R\leq|x|)}^{(p+1)\theta_q}\\
		& \leq C\|u(t)\|_{H^1}^{(p+1)(1-\theta_q)}\|u(t)\|_{L_x^2(R\leq|x|)}^{(p+1)\theta_q}
		\leq C\|u(t)\|_{L_x^2(R\leq|x|)}^{(p+1)\theta_q}.
\end{align*}
Next, we estimate $\mathcal{R}_3$.
\begin{align*}
	\mathcal{R}_3
		&= -\int_{R\leq|x|\leq3R}\left\{\frac{1}{R^2}\mathscr{X}^{(4)}\left(\frac{r}{R}\right)+\frac{2(d-1)}{Rr}\mathscr{X}^{(3)}\left(\frac{r}{R}\right)+\frac{(d-1)(d-3)}{r^2}\mathscr{X}''\left(\frac{r}{R}\right)\right.\\
		&\hspace{5.0cm} \left.+\frac{(d-1)(3-d)R}{r^3}\mathscr{X}'\left(\frac{r}{R}\right)\right\}|u(t,x)|^2dx\\
		&\leq \frac{C}{R^2}\|u(t)\|_{L_x^2(R\leq|x|)}^2
		\leq \frac{C}{R^2}.
\end{align*}
Finally, $\mathcal{R}_4$ is estimated as $\mathcal{R}_4 \leq 0$,
which completes the proof of the lemma.
\end{proof}

\begin{lemma}\label{Lemma3 for blows-up or grows-up}
Let $d \geq 1$, $2_\ast < p+1 < 2^\ast$, $\gamma > 0$, and $0 < \mu < \min\{2,d\}$.
If $u_0 \in PW_{-,3}$, then there exists $\delta > 0$ such that $K_\gamma(u(t)) < -\delta$ for any $t \in (T_\text{min},T_\text{max})$.
\end{lemma}

\begin{proof}
We note that $u(t)$ satisfies \eqref{167} for any $t \in (T_\text{min},T_\text{max})$ by Lemma \ref{Coercivity 1}.
The virial functional can be written as follows:
\begin{align}
	K_\gamma(u(t))
		= d(p-1)E_\gamma[u_0] - \frac{dp-(d+4)}{2}\|(-\Delta_\gamma)^\frac{1}{2}u(t)\|_{L_x^2}^2 + (\mu-2)\int_{\mathbb{R}^d}\frac{\gamma}{|x|^\mu}|u(t,x)|^2dx. \label{065}
\end{align}
By the assumption \eqref{165},
\begin{align*}
	\varepsilon_1
		:=\frac{1}{2}\left\{\left(\frac{M[Q_{1,0}]}{M[u_0]}\right)^{\frac{1-s_c}{s_c}}E_0[Q_{1,0}]-E_\gamma[u_0]\right\}>0
\end{align*}
and
\begin{align}
	E_\gamma[u_0]
		< \frac{1}{2}E_\gamma[u_0]+\frac{1}{2}\left(\frac{M[Q_{1,0}]}{M[u_0]}\right)^\frac{1-s_c}{s_c}E_0[Q_{1,0}]
		= \left(\frac{M[Q_{1,0}]}{M[u_0]}\right)^\frac{1-s_c}{s_c}E_0[Q_{1,0}]-\varepsilon_1. \label{021}
\end{align}
Moreover, we have
\begin{align}
	\|(-\Delta_\gamma)^\frac{1}{2}u(t)\|_{L_x^2}^2
		> \left(\frac{M[Q_{1,0}]}{M[u_0]}\right)^\frac{1-s_c}{s_c}\|\nabla Q_{1,0}\|_{L_x^2}^2
		= \frac{2d(p-1)}{dp-(d+4)}\left(\frac{M[Q_{1,0}]}{M[u_0]}\right)^\frac{1-s_c}{s_c}E_0[Q_{1,0}] \label{066}
\end{align}
for any $t\in(T_\text{min},T_\text{max})$ by the estimate \eqref{167} and \eqref{111}.
Therefore, \eqref{065}, \eqref{021}, and \eqref{066} give
\begin{align*}
	K_\gamma(u(t))
		& < -d(p-1)\varepsilon_1
		= :-\delta.
\end{align*}
\end{proof}

\begin{proof}[Proof of blow-up or grow-up in Theorem \ref{Global versus blow-up or grow-up} and Theorem \ref{Radial blow-up}]
We assume that
\begin{align*}
	T_\text{max}
		= \infty\ \ \text{ and }\ \ 
	\sup_{t\in[0,\infty)}\|\nabla u(t)\|_{L_x^2}
		< \infty
\end{align*}
for contradiction.
By Lemma \ref{Coercivity 2} and Lemma \ref{Lemma3 for blows-up or grows-up}, there exists $\delta>0$ such that $K_\gamma(u(t)) < -\delta$
for any $t \in [0,\infty)$.
We consider the function $I(t)$ as Lemma \ref{Lemma2 for blows-up or grows-up}.
From Lemma \ref{Lemma2 for blows-up or grows-up} and Lemma \ref{Lemma1 for blows-up or grows-up}, we have
\begin{align}
	I''(s)
		&\leq -4\delta + C\,\|u(s)\|_{L_x^2(R\leq|x|)}^{(p+1)\theta_q} + \frac{C}{R^2}
		\leq -4\delta + C\eta^\frac{(p+1)\theta_q}{2}+o_R(1)\label{067}
\end{align}
for any $\eta > 0$, $R > 0$, and $s \in \left[0,\frac{\eta R}{6C_0\|u_0\|_{L_x^2}}\right]$.
We take $\eta=\eta_0>0$ sufficiently small such as
\begin{align*}
	C\eta_0^\frac{(p+1)\theta_q}{2}
		\leq 2\delta.
\end{align*}
Then, \eqref{067} implies
\begin{align}
	I''(s)
		\leq -2\delta + o_R(1) \label{068}
\end{align}
for any $R > 0$ and $s \in \left[0,\frac{\eta_0R}{6C_0\|u_0\|_{L_x^2}}\right]$.
We set
\begin{align*}
	T
		= T(R)
		:= \alpha_0R
		:= \frac{\eta_0R}{6C_0\|u_0\|_{L_x^2}}.
\end{align*}
Integrating \eqref{068} over $s\in[0,t]$ and integrating over $t\in[0,T]$, we have
\begin{align}
	I(T)
		\leq I(0) + I'(0)T + \frac{1}{2}\left(-2\delta+o_R(1)\right)T^2
		= I(0) + I'(0)\alpha_0R + \frac{1}{2}\left(-2\delta+o_R(1)\right)\alpha_0^2R^2. \label{070}
\end{align}
Here, we can prove
\begin{align}
	I(0)
		=o_R(1)R^2\ \ \text{ and }\ \ 
	I'(0)
		=o_R(1)R. \label{071}
\end{align}
Indeed,
\begin{align*}
	I(0)
		&=\int_{|x|\leq3R}\mathscr{X}_R(x)|u_0(x)|^2dx
		=\int_{|x|\leq\sqrt{R}}|x|^2|u_0(x)|^2dx + \int_{\sqrt{R}\leq|x|\leq3R}R^2\mathscr{X}\left(\frac{r}{R}\right)|u_0(x)|^2dx\\
		&\leq RM[u_0] + c R^2\int_{\sqrt{R}\leq|x|}|u_0(x)|^2dx
		=o_R(1)R^2,
\end{align*}
and
\begin{align*}
	I'(0)
		&= 2\text{Im}\int_{\mathbb{R}^d}R\mathscr{X}'\left(\frac{r}{R}\right)\frac{x\cdot\nabla u_0}{r}\overline{u_0}dx\\
		&= 4\text{Im}\int_{|x|\leq\sqrt{R}}x\cdot\nabla u_0\overline{u_0}dx + 2\text{Im}\int_{\sqrt{R}\leq|x|\leq3R}R\mathscr{X}'\left(\frac{r}{R}\right)\frac{x\cdot\nabla u_0}{r}\overline{u_0}dx\\
		&\leq 4\sqrt{R}\,\|\nabla u_0\|_{L_x^2}\|u_0\|_{L_x^2} + c R\|\nabla u_0\|_{L_x^2}\|u_0\|_{L_x^2(\sqrt{R}\leq|x|)}
		= o_R(1)R.
\end{align*}
Combining \eqref{070} and \eqref{071}, we get
\begin{align*}
	I(T)
		\leq (o_R(1)-\delta\alpha_0^2)R^2.
\end{align*}
We take $R>0$ sufficiently large such as $o_R(1) - \delta\alpha_0^2 < 0$.
However, this is contradiction to
\begin{align*}
	I(T)
		= \int_{\mathbb{R}^d}\mathscr{X}_R(x)|u(T,x)|^2dx
		\geq 0.
\end{align*}
\end{proof}

Using Theorem \ref{Dinh182} and Theorem \ref{Global versus blow-up or grow-up}, we prove Theorem \ref{Equivalence of L^2H^1 and K}.

\begin{proof}[Proof of Theorem \ref{Equivalence of L^2H^1 and K}]
We note that
\begin{align*}
	PW_{+,\,j}\cup PW_{-,\,j}
		= \{u_0 \in H^1(\mathbb{R}^d):\eqref{165}\}
\end{align*}
for any $j = 1, 2, 3$.
If $u_0 \in PW_{+,\,j}$ ($j = 1, 2, 3$), then a solution $u$ to \eqref{NLS} is uniformly bounded in $H^1(\mathbb{R}^d)$.
On the other hand, if $u_0 \in PW_{-,\,j}$ ($j = 1, 2, 3$), then a solution $u$ to \eqref{NLS} is unbounded in $H^1(\mathbb{R}^d)$.
\end{proof}

To complete this section, we prove Corollary \ref{Negative energy} by Theorem \ref{Equivalence of L^2H^1 and K}.

\begin{proof}[Proof of Corollary \ref{Negative energy}]
Let $E_\gamma[u_0] \leq 0$ and $u_0 \neq 0$.
\eqref{165} holds clearly.
$E_\gamma[u_0] \leq 0$ implies
\begin{align*}
	\frac{1}{2}\|(-\Delta_\gamma)^\frac{1}{2}u_0\|_{L^2}^2
		\leq \frac{1}{p+1}\|u_0\|_{L^{p+1}}^{p+1},
\end{align*}
so we have
\begin{align*}
	K_\gamma(u_0)
		& < 2\|(-\Delta_\gamma)^\frac{1}{2}u_0\|_{L^2}^2 - \frac{d(p-1)}{p+1}\|u_0\|_{L^{p+1}}^{p+1}
		\leq \frac{d+4-dp}{2}\|(-\Delta_\gamma)^\frac{1}{2}u_0\|_{L^2}^2
		< 0.
\end{align*}
\end{proof}

\section{Blow-up}\label{Blow-up}

In this section, we prove the blow-up results in Theorem \ref{Global versus blow-up or grow-up} and Theorem \ref{Radial blow-up}.
This proof is based on \cite{Gla77} and \cite{OgaTsu91} (see also \cite{HolRou08}).
As Section  \ref{Blow-up or grow-up}, we consider only positive time direction.
First, we prove the following lemma to get the blow-up results.

\begin{lemma}[Another characterization of $n_{\omega,\gamma}$ and $r_{\omega,\gamma}$]\label{Another characterization of n and r}
Let $d \geq 1$, $2_\ast < p+1 < 2^\ast$, $\gamma > 0$, and $0 < \mu < \min\{2,d\}$.
Let $\omega > 0$.
Then, we have
\begin{align*}
	n_{\omega,\gamma}
		= \inf\{T_{\omega,\gamma}(f):f\in H^1(\mathbb{R}^d)\setminus\{0\},\ K_\gamma(f) \leq 0\}
\end{align*}
and
\begin{align*}
	r_{\omega,\gamma}
		= \inf\{T_{\omega,\gamma}(f):f\in H_\text{rad}^1(\mathbb{R}^d)\setminus\{0\},\ K_\gamma(f) \leq 0\},
\end{align*}
where $T_{\omega,\gamma}$ is defined as
\begin{align*}
	T_{\omega,\gamma}(f)
		:= S_{\omega,\gamma}(f) - \frac{1}{d(p-1)}K_\gamma(f).
\end{align*}
\end{lemma}

\begin{proof}
We consider only $n_{\omega,\gamma}$ since the case of $r_{\omega,\gamma}$ holds by the same manner.
We take any $f \in H^1(\mathbb{R}^d)\setminus\{0\}$ satisfying $K_\gamma(f) \leq 0$.
There exists $0 < \lambda \leq 1$ such that $K_\gamma(\lambda f) = 0$, so
\begin{align*}
	n_{\omega,\gamma}
		\leq S_{\omega,\gamma}(\lambda f)
		= T_{\omega,\gamma}(\lambda f)
		\leq T_{\omega,\gamma}(f).
\end{align*}
Therefore, we obtain
\begin{align*}
	n_{\omega,\gamma}
		& = \inf\{S_{\omega,\gamma}(f):f \in H^1(\mathbb{R}^d),\ K_\gamma(f) = 0\}\\
		& \geq \inf\{S_{\omega,\gamma}(f):f \in H^1(\mathbb{R}^d),\ K_\gamma(f) \leq 0\}
		\geq n_{\omega,\gamma},
\end{align*}
which complete proof of this lemma.
\end{proof}

\begin{proof}[Proof of blow-up in Theorem \ref{Global versus blow-up or grow-up} and Theorem \ref{Radial blow-up}]
　\\
($u_0 \in |x|^{-1}L^2(\mathbb{R}^d)$ case)\\
When $u_0 \in |x|^{-1}L^2(\mathbb{R}^d)$, there exists a positive constant $\delta > 0$ such that
\begin{align*}
	\frac{d^2}{dt^2}\|xu(t)\|_{L_x^2}^2
		= 4K_\gamma(u(t))
		< -\delta
\end{align*}
for any $t \in (T_\text{min}, T_\text{max})$ by \eqref{168}, Lemma \ref{Coercivity 2} and Lemma \ref{Lemma3 for blows-up or grows-up}.
This inequality implies the desired result.\\ 
($u_0 \in H_\text{rad}^1(\mathbb{R}^d)$ case)\\
We consider a functional $I$ in Lemma \ref{Lemma2 for blows-up or grows-up}.
\begin{align*}
	I''(t)
		= 4K_\gamma(u) + \mathcal{R}_1 + \mathcal{R}_2 + \mathcal{R}_3 + \mathcal{R}_4,
\end{align*}
where $\mathcal{R}_1$, $\mathcal{R}_2$, $\mathcal{R}_3$, and $\mathcal{R}_4$ are defined as \eqref{160}, \eqref{161}, \eqref{162}, and \eqref{163} respectively.
We have already gotten $R_1 \leq 0$, $R_3 \leq \frac{C}{R^2}$, and $R_4 \leq 0$ in Lemma \ref{Lemma2 for blows-up or grows-up}.
We estimate $\mathcal{R}_2$.
\begin{align*}
	\mathcal{R}_2
		&\leq c\|u\|_{L_x^{p+1}(R\leq|x|)}^{p+1}
		\leq \frac{c}{R^\frac{(d-1)(p-1)}{2}\varepsilon}M[u]^\frac{p+3}{4}\cdot\varepsilon\|\nabla f\|_{L_x^2(R\leq|x|)}^\frac{p-1}{2}\\
		&\leq 
\begin{cases}
&\hspace{-0.4cm}\displaystyle{\frac{c}{R^2}\|\nabla f\|_{L_x^2}^2,\quad(d=2,\ p=5),}\\
&\hspace{-0.4cm}\displaystyle{\frac{c}{R^\frac{2(d-1)(p-1)}{5-p}\varepsilon^\frac{4}{5-p}}M[u]^\frac{p+3}{5-p} + 2\{d(p-1)-4\}\varepsilon\|\nabla f\|_{L_x^2}^2,\quad(\text{otherwise})}
\end{cases}
\end{align*}
by Lemma \ref{Radial Sobolev inequality} and Young's inequality.
Let $0 < \varepsilon < \frac{2d(p-1)-4\mu}{2d(p-1)-8}$.
We take a positive constant $\delta >0$ such as $S_{\omega,\gamma}(u) < (1-\delta)m_{\omega,\gamma}$.
Since $m_{\omega,\gamma} \leq T_{\omega,\gamma}(u)$ by Lemma \ref{Another characterization of n and r}, we have
\begin{align*}
	I''(t)
		& \leq 4K_\gamma(u) + \frac{c}{R^\frac{2(d-1)(p-1)}{5-p}\varepsilon^\frac{4}{5-p}}M[u]^\frac{p+3}{5-p} + 2\{d(p-1)-4\}\varepsilon\|\nabla f\|_{L^2}^2 + \frac{C}{R^2}\\
		& < 4d(p-1)S_{\omega,\gamma}(u) - 2\omega d(p-1)M[u] - 2(1-\varepsilon)\{d(p-1)-4\}\|\nabla f\|_{L^2}^2 + \frac{C}{R^2}\\
		&\hspace{1.0cm} - \{2d(p-1)(1-\varepsilon)+4(2\varepsilon-\mu)\}\int_{\mathbb{R}^d}\frac{\gamma}{|x|^\mu}|u(x)|^2dx + \frac{c}{R^\frac{2(d-1)(p-1)}{5-p}\varepsilon^\frac{4}{5-p}}M[u]^\frac{p+3}{5-p}\\
		& < 4d(p-1)S_{\omega,\gamma}(u) - 4d(p-1)(1-\varepsilon)T_{\omega,\gamma}(u) + \frac{C}{R^2} + \frac{c}{R^\frac{2(d-1)(p-1)}{5-p}\varepsilon^\frac{4}{5-p}}M[u]^\frac{p+3}{5-p}\\
		& < 4d(p-1)(1-\delta)m_{\omega,\gamma} - 4d(p-1)(1-\varepsilon)m_{\omega,\gamma} + \frac{C}{R^2} + \frac{c}{R^\frac{2(d-1)(p-1)}{5-p}\varepsilon^\frac{4}{5-p}}M[u]^\frac{p+3}{5-p}\\
		& = 4d(p-1)(\varepsilon-\delta)m_{\omega,\gamma} + \frac{C}{R^2} + \frac{c}{R^\frac{2(d-1)(p-1)}{5-p}\varepsilon^\frac{4}{5-p}}M[u]^\frac{p+3}{5-p}
\end{align*}
for $d \geq 2$ and $p < 5$.
Taking $\frac{c}{R^2} \leq 2\{d(p-1)-4\}\varepsilon$, we have
\begin{align*}
	I''(t)
		< 4d(p-1)(\varepsilon-\delta)m_{\omega,\gamma} + \frac{C}{R^2}
\end{align*}
for $d = 2$ and $p = 5$ by the same manner.
Thus, if we take sufficiently small $0 < \varepsilon < \min\{\delta,\frac{2d(p-1)-4\mu}{2d(p-1)-8}\}$ and sufficiently large $R > 0$, then we obtain $I''(t)<0$.
This implies $u$ blows up.
\end{proof}

\section{Appendix A}\label{Appendix A}

In this section, we check some properties of $n_{\omega,\gamma}^{\alpha,\beta}$.
More precisely, we prove Proposition \ref{Minimization problem} with $V(x) = \frac{\gamma}{|x|^\gamma}$ and Proposition \ref{Equivalence of ME and S<m} with $V(x) = \frac{\gamma}{|x|^\mu}$ for convenience of the readers.
For the proof, see also \cite{IkeInu17}.
We define the following functional:
\begin{align*}
	&U_{\omega,\gamma}^{\alpha,\beta}(f):
		= S_{\omega,\gamma}(f)-\frac{1}{2\alpha-(d-2)\beta}K_{\omega,\gamma}^{\alpha,\beta}(f)\\
		& = \frac{\beta\omega}{2\alpha-(d-2)\beta}\|f\|_{L^2}^2+\frac{(2-\mu)\beta}{2\{2\alpha-(d-2)\beta\}}\int_{\mathbb{R}^d}\frac{\gamma}{|x|^\mu}|f(x)|^2dx+\frac{(p-1)\alpha-2\beta}{(p+1)\{2\alpha-(d-2)\beta\}}\|f\|_{L^{p+1}}^{p+1}.
\end{align*}

\begin{lemma}\label{Rewriting n}
Let $d \geq 1$, $2_\ast < p+1 < 2^\ast$, $\gamma \geq 0$, and $0 < \mu < \min\{2,d\}$.
Let $(\alpha,\beta)$ satisfy \eqref{104} and $\omega > 0$.
Then,
\begin{align*}
	n_{\omega,\gamma}^{\alpha,\beta}
		= \inf\left\{U_{\omega,\gamma}^{\alpha,\beta}(f):f\in H^1(\mathbb{R}^d)\setminus\{0\},\ K_{\omega,\gamma}^{\alpha,\beta}(f)\leq0\right\}
\end{align*}
holds.
\end{lemma}

\begin{proof}
This lemma follows from the same argument with Lemma \ref{Another characterization of n and r}.
\end{proof}

\begin{proposition}
Let $d \geq 1$, $2_\ast < p+1 < 2^\ast$, $\gamma > 0$, and $0 < \mu < \min\{2,d\}$.
Let $(\alpha,\beta)$ satisfy \eqref{104} and $\omega > 0$.
Then, $n_{\omega,\gamma}^{\alpha,\beta} = n_{\omega,0}$ holds.
\end{proposition}

\begin{proof}
First, we prove $n_{\omega,\gamma}^{\alpha,\beta} \geq n_{\omega,0}$.
We take any $f \in H^1(\mathbb{R}^d)\setminus \{0\}$ with $K_{\omega,\gamma}^{\alpha,\beta}(f)=0$.
$K_{\omega,0}^{\alpha,\beta}(f) \leq K_{\omega,\gamma}^{\alpha,\beta}(f) = 0$, so 
\begin{align*}
	n_{\omega,0}
		\leq U_{\omega,0}^{\alpha,\beta}(f)
		\leq U_{\omega,\gamma}^{\alpha,\beta}(f)
		= S_{\omega,\gamma}(f).
\end{align*}
This inequality implies $n_{\omega,\gamma}^{\alpha,\beta} \geq n_{\omega,0}$.
Next, we prove $n_{\omega,\gamma}^{\alpha,\beta} \leq n_{\omega,0}$.
We note that $Q_{\omega,0}$ attains $n_{\omega,0}$, that is, $S_{\omega,0}(Q_{\omega,0}) = n_{\omega,0}$ and $K_{\omega,0}^{\alpha,\beta}(Q_{\omega,0}) = 0$.
We take any sequence $\{y_n\}\subset\mathbb{R}^d$ satisfying $|y_n|\longrightarrow\infty$ as $n \rightarrow \infty$.
Then, we have
\begin{align*}
	S_{\omega,\gamma}(Q_{\omega,0}(\,\cdot\,-y_n))
		\longrightarrow S_{\omega,0}(Q_{\omega,0})\ \ \text{ as }\ \ n\rightarrow\infty
\end{align*}
and
\begin{align*}
	K_{\omega,\gamma}^{\alpha,\beta}(Q_{\omega,0}(\,\cdot\,-y_n))
		> K_{\omega,0}^{\alpha,\beta}(Q_{\omega,0}(\,\cdot\,-y_n))
		= K_{\omega,0}^{\alpha,\beta}(Q_{\omega,0})
		= 0.
\end{align*}
Since $K_{\omega,\gamma}^{\alpha,\beta}(Q_{\omega,0}(\,\cdot\,-y_n))>0$ and $K_{\omega,\gamma}^{\alpha,\beta}(\lambda Q_{\omega,0}(\,\cdot\,-y_n))<0$ for sufficiently large $\lambda>1$, there exists $\{\lambda_n\}$ such that $K_{\omega,\gamma}^{\alpha,\beta}(\lambda_nQ_{\omega,0}(\,\cdot\,-y_n))=0$.
Then, $K_{\omega,\gamma}^{\alpha,\beta}(\lambda_nQ_{\omega,0}(\,\cdot\,-y_n))=0$
and $K_{\omega,0}^{\alpha,\beta}(Q_{\omega,0})=0$ imply
\begin{align*}
	\frac{2\alpha-(d-\mu)\beta}{2}\int_{\mathbb{R}^d}\frac{\gamma}{|x|^\mu}|Q_{\omega,0}(x-y_n)|^2dx
		=\frac{(p+1)\alpha-d\beta}{p+1}(\lambda_n^{p-1}-1)\|Q_{\omega,0}\|_{L^{p+1}}^{p+1}.
\end{align*}
The left hand side goes to zero as $n\rightarrow\infty$, so the right hand side goes to zero as $n\rightarrow\infty$, that is, $\displaystyle \lim_{n\rightarrow\infty}\lambda_n=1$.
Therefore, $S_{\omega,\gamma}(\lambda_nQ_{\omega,0}(\,\cdot\,-y_n))\longrightarrow S_{\omega,0}(Q_{\omega,0}) = n_{\omega,0}$ as $n\rightarrow\infty$.
Combining this fact and $K_{\omega,\gamma}^{\alpha,\beta}(\lambda_nQ_{\omega,\gamma}(\,\cdot\,-y_n))=0$ for each $n\in\mathbb{N}$, we obtain $n_{\omega,\gamma}^{\alpha,\beta} \leq n_{\omega,0}$.
\end{proof}



\begin{proposition}
Let $d \geq 1$, $2_\ast < p+1 < 2^\ast$, $\gamma > 0$, and $0 < \mu < \min\{2,d\}$.
Let $(\alpha,\beta)$ satisfy \eqref{104} and $\omega > 0$.
Then, $n_{\omega,\gamma}^{\alpha,\beta}$ is not attained.
\end{proposition}

\begin{proof}
We assume for contradiction that $\phi$ attains $n_{\omega,\gamma}^{\alpha,\beta}$, that is, $S_{\omega,\gamma}(\phi) = n_{\omega,\gamma}^{\alpha,\beta}$ and $K_{\omega,\gamma}^{\alpha,\beta}(\phi)=0$.
There exists $y\in \mathbb{R}^d$ such that
\begin{align*}
	K_{\omega,\gamma}^{\alpha,\beta}(\phi(\,\cdot\,-y))
		< K_{\omega,\gamma}^{\alpha,\beta}(\phi)
		= 0.
\end{align*}
Since $K_{\omega,\gamma}^{\alpha,\beta}(\lambda\phi(\,\cdot\,-y))>0$ for sufficient small $\lambda\in (0,1)$ and $K_{\omega,\gamma}^{\alpha,\beta}(\phi(\,\cdot\,-y))<0$, there exists $\lambda_0\in (0,1)$ such that
\begin{align*}
K_{\omega,\gamma}^{\alpha,\beta}(\lambda_0\phi(\,\cdot\,-y))=0.
\end{align*}
Therefore, we obtain
\begin{align*}
	n_{\omega,\gamma}^{\alpha,\beta}
		\leq S_{\omega,\gamma}(\lambda_0\phi(\,\cdot\,-y))
		= U_{\omega,\gamma}^{\alpha,\beta}(\lambda_0\phi(\,\cdot\,-y))
		< U_{\omega,\gamma}^{\alpha,\beta}(\phi(\,\cdot\,-y))
		< U_{\omega,\gamma}^{\alpha,\beta}(\phi)
		= S_{\omega,\gamma}(\phi)
	 	= n_{\omega,\gamma}^{\alpha,\beta}.
\end{align*}
This is contradiction.
\end{proof}

\begin{proof}[Proof of Proposition \ref{Equivalence of ME and S<m}]
By the equation \eqref{SP},
it follows that
$Q_{\omega,0}=\omega^\frac{1}{p-1}Q_{1,0}(\omega^\frac{1}{2}\,\cdot\,)$ holds.
Then,
\begin{align*}
	S_{\omega,0}(Q_{\omega,0})
		&=\omega^{\frac{d+2-(d-2)p}{2(p-1)}}S_{1,0}(Q_{1,0})
\end{align*}
holds.
Thus, the condition (2) is equivalent to $S_{\omega,\gamma}(u_0)<\omega^{\frac{d+2-(d-2)p}{2(p-1)}}S_{1,0}(Q_{1,0})$.
Here, we define a function $f(\omega):=\omega^{\frac{d+2-(d-2)p}{2(p-1)}}S_{1,0}(Q_{1,0})-S_{\omega,\gamma}(u_0)$ on $\omega\in (0,\infty)$.
The function $f$ has a maximum value at $\omega=\omega_0$ by $p>1+\frac{4}{d}$.
Therefore, if there exists $\omega>0$ such that $S_{\omega,\gamma}(u_0)<S_{\omega,0}(Q_{\omega,0})$, then $f(\omega_0)>0$ holds.
$f(\omega_0)>0$ implies
\begin{align*}
	\frac{d+2-(d-2)p}{p-1}\left[\frac{dp-(d+4)}{2\{d+2-(d-2)p\}}\right]^\frac{dp-(d+4)}{2(p-1)}S_{1,0}(Q_{1,0})
		>M[u_0]^{1-s_c}E_\gamma[u_0]^{s_c}.
\end{align*}
By using Proposition \ref{Pohozaev identity} and \eqref{111},
we obtain
\begin{align*}
	M[Q_{1,0}]^{1-s_c}E_0[Q_{1,0}]^{s_c}
		>M[u_0]^{1-s_c}E_\gamma[u_0]^{s_c}.
\end{align*}
\end{proof}

\section{Appendix B}\label{Appendix B}

In this section, we check some properties of $r_{\omega,\gamma}^{\alpha,\beta}$.
More precisely, we prove Proposition \ref{Existence of a radial ground state} with $V(x) = \frac{\gamma}{|x|^\mu}$ for convenient of the readers.
For the proof, see also \cite{IbrMasNak11} and \cite{IkeInu17}.



\begin{proposition}[Equivalence of $H^1$-norm and $S_{\omega,\gamma}$]\label{Equivalence of H1 and S}
Let $d \geq 1$, $2_\ast < p+1 < 2^\ast$, $\gamma > 0$, and $0 < \mu < \min\{2,d\}$.
Let $(\alpha,\beta)$ satisfy \eqref{104} and $\omega > 0$.
If $f \in H^1(\mathbb{R}^d)$ satisfies $K_{\omega,\gamma}^{\alpha,\beta}(f) \geq 0$, then
\begin{align*}
	\{(p-1)\alpha-2\beta\}S_{\omega,\gamma}(f)
		\leq \frac{(p-1)\alpha-2\beta}{2}J_{\omega,\gamma}(f)
		\leq \{(p+1)\alpha-d\beta\}S_{\omega,\gamma}(f)
\end{align*}
holds, where $J_{\omega,\gamma}(f):=\omega\|f\|_{L^2}^2+\|(-\Delta_\gamma)^\frac{1}{2}f\|_{L^2}^2$.
\end{proposition}

\begin{proof}
The first inequality holds clearly.
We see the second inequality by the following relation:
\begin{align*}
	&\frac{(p-1)\alpha-2\beta}{2}J_{\omega,\gamma}(f)
		\leq \frac{(p-1)\alpha-2\beta}{2}J_{\omega,\gamma}(f)+K_{\omega,\gamma}^{\alpha,\beta}(f)
		\leq \{(p+1)\alpha-d\beta\}S_{\omega,\gamma}(f).
\end{align*}
\end{proof}



\begin{lemma}\label{Rewriting r}
Let $d \geq 1$, $2_\ast < p+1 < 2^\ast$, $\gamma > 0$, and $0 < \mu < \min\{2,d\}$.
Let $(\alpha,\beta)$ satisfy \eqref{104} and $\omega > 0$.
If $f\in H_\text{rad}^1(\mathbb{R}^d)\setminus\{0\}$ satisfies $K_{\omega,\gamma}^{\alpha,\beta}(f) \leq 0$, then there exists $0 < \lambda \leq 1$ such that
\begin{align*}
	K_{\omega,\gamma}^{\alpha,\beta}(\lambda f)
		= 0\ \ \text{ and }\ \ 
	r_{\omega,\gamma}^{\alpha,\beta}
		\leq S_{\omega,\gamma}(\lambda f)
		= U_{\omega,\gamma}^{\alpha,\beta}(\lambda f)
		\leq U_{\omega,\gamma}^{\alpha,\beta}(f).
\end{align*}
In particular,
\begin{align*}
	r_{\omega,\gamma}^{\alpha,\beta}
		= \inf\left\{U_{\omega,\gamma}^{\alpha,\beta}(f):f\in H_\text{rad}^1(\mathbb{R}^d)\setminus\{0\},\ K_{\omega,\gamma}^{\alpha,\beta}(f)\leq0\right\}
\end{align*}
holds.
\end{lemma}

\begin{proof}
This lemma follows from the same argument with Lemma \ref{Another characterization of n and r}.
\end{proof}

\begin{proposition}
Let $d \geq 2$, $2_\ast < p < 2^\ast$, $\gamma > 0$, and $0 < \mu < 2$.
Let $(\alpha,\beta)$ satisfies \eqref{104} and $\omega > 0$.
Then, $r_{\omega,\gamma}^{\alpha,\beta}$ is attained.
\end{proposition}

\begin{proof}
We take a minimizing sequence $\{f_n\}\subset H_\text{rad}^1(\mathbb{R}^d)\setminus\{0\}$ satisfying
\begin{align}
	K_{\omega,\gamma}^{\alpha,\beta}(f_n)
		=0\ \text{ for any }\ n\in\mathbb{N} \label{117}
\end{align}
and
\begin{align*}
	S_{\omega,\gamma}(f_n)
		= U_{\omega,\gamma}^{\alpha,\beta}(f_n)
		\searrow r_{\omega,\gamma}^{\alpha,\beta}\ \text{ as }\ n\rightarrow\infty.
\end{align*}
$\{f_n\}$ is a bounded sequence in $H^1(\mathbb{R}^d)$ by Lemma \ref{Equivalence of H1 and S}.
We can take a subsequence of $\{f_n\}$ satisfying $f_n\xrightharpoonup[]{\hspace{0.4cm}}f_\infty$ in $H^1(\mathbb{R}^d)$, which is denoted still by $\{f_n\}$.
Since $H_\text{rad}^1(\mathbb{R}^d) \subset L^{p+1}(\mathbb{R}^d)$ is compactly embedding for $2_\ast < p+1 < 2^\ast$, we can take a subsequence of $\{f_n\}$ satisfying $f_n\longrightarrow f_\infty$ in $L^{p+1}(\mathbb{R}^d)$, which is denoted still by $\{f_n\}$.
Therefore, we get
\begin{align*}
	\|f_\infty\|_{L^2}
		\leq \liminf_{n\rightarrow\infty}\|f_n\|_{L^2},\ \ 
	\|(-\Delta_\gamma)^\frac{1}{2}f_\infty\|_{L^2}
		\leq \liminf_{n\rightarrow\infty}\|(-\Delta_\gamma)^\frac{1}{2}f_n\|_{L^2},
\end{align*}
\begin{align}
	\|f_\infty\|_{L^{p+1}}^{p+1}
		= \lim_{n\rightarrow\infty}\|f_n\|_{L^{p+1}}^{p+1}, \label{118}
\end{align}
where we use the following estimate to show the second inequality:
\begin{align*}
	\int_{\mathbb{R}^d}\frac{\gamma}{|x|^\mu}|f_\infty(x)|^2dx
		\leq \liminf_{n\rightarrow \infty}\int_{\mathbb{R}^d}\frac{\gamma}{|x|^\mu}|f_n(x)|^2dx,
\end{align*}
which is proved by the same argument with lower semi-continuity for weak convergence.
These inequalities deduce
\begin{align*}
	S_{\omega,\gamma}(f_\infty)
		\leq \liminf_{n\rightarrow\infty}S_{\omega,\gamma}(f_n)
		= r_{\omega,\gamma}^{\alpha,\beta}
\end{align*}
and
\begin{align}
	U_{\omega,\gamma}^{\alpha,\beta}(f_\infty)
		\leq \liminf_{n\rightarrow\infty}U_{\omega,\gamma}^{\alpha,\beta}(f_n)
		= r_{\omega,\gamma}^{\alpha,\beta}. \label{119}
\end{align}
First, we prove that $f_\infty$ is not trivial.
We assume $f_\infty=0$ for contradiction.
Then, we have
\begin{align*}
	0
		&=\frac{(p+1)\alpha-d\beta}{p+1}\lim_{n\rightarrow\infty}\|f_n\|_{L^{p+1}}^{p+1}
		\geq \frac{2\alpha-(d-2)\beta}{2}\liminf_{n\rightarrow\infty}\|\nabla f_n\|_{L^2}^2
		\geq0
\end{align*}
by \eqref{118} and \eqref{117}, that is, $\displaystyle \liminf_{n\rightarrow\infty}\|\nabla f_n\|_{L^2}=0$.
From Lemma \ref{Positivity of K}, we get $K_{\omega,\gamma}^{\alpha,\beta}(f_n)>0$ for sufficiently large $n$.
This is contradiction with \eqref{117}.
Next, we prove that there exists a function $Q_{\omega,\gamma} \in H_\text{rad}^1\setminus\{0\}$ such that $S_{\omega,\gamma}(Q_{\omega,\gamma}) = r_{\omega,\gamma}^{\alpha,\beta}$ and $K_{\omega,\gamma}^{\alpha,\beta}(Q_{\omega,\gamma})=0$.
Using the following Lemma \ref{Brezis-Lieb}:
\begin{lemma}[Brezis--Lieb, \cite{BreLie83}]\label{Brezis-Lieb}
The following holds.
\begin{itemize}
\item[(1)] Let $1<q<\infty$.
We assume that $\{u_n\}$ is bounded in $L^q$ and satisfies $u_n\longrightarrow u$ a.e. in $\mathbb{R}^d$.
Then, we have $u\in L^q$ and
\begin{align*}
	\lim_{n\rightarrow\infty}\left(\|u_n\|_{L^q}^q-\|u_n-u\|_{L^q}^q\right)
		=\|u\|_{L^q}^q.
\end{align*}
\item[(2)] Let $X$ be a Hilbert space. If $u_n\xrightharpoonup[]{\hspace{0.4cm}}u$ in $X$, then we have
\begin{align*}
\lim_{n\rightarrow\infty}\left(\|u_n\|_{X}^2-\|u_n-u\|_{X}^2\right)=\|u\|_{X}^2.
\end{align*}
\end{itemize}
\end{lemma}
, we have
\begin{align*}
	S_{\omega,\gamma}(f_n)-S_{\omega,\gamma}(f_n-f_\infty)
		\longrightarrow S_{\omega,\gamma}(f_\infty)\ \ \text{ as }\ \ n\rightarrow\infty,
\end{align*}
\begin{align*}
	U_{\omega,\gamma}^{\alpha,\beta}(f_n) - U_{\omega,\gamma}^{\alpha,\beta}(f_n-f_\infty)
		\longrightarrow U_{\omega,\gamma}^{\alpha,\beta}(f_\infty)\ \ \text{ as }\ \ n\rightarrow\infty,
\end{align*}
\begin{align*}
	- K_{\omega,\gamma}^{\alpha,\beta}(f_n-f_\infty)
		\longrightarrow K_{\omega,\gamma}^{\alpha,\beta}(f_\infty)\ \ \text{ as }\ \ n\rightarrow\infty.
\end{align*}
Since
\begin{align*}
	\lim_{n\rightarrow\infty}U_{\omega,\gamma}^{\alpha,\beta}(f_n-f_\infty)
		= \lim_{n\rightarrow\infty}U_{\omega,\gamma}^{\alpha,\beta}(f_n)-U_{\omega,\gamma}^{\alpha,\beta}(f_\infty)
		= r_{\omega,\gamma} - U_{\omega,\gamma}^{\alpha,\beta}(f_\infty)
		< r_{\omega,\gamma},
\end{align*}
there exists $n_0\in \mathbb{N}$ such that
\begin{align*}
	K_{\omega,\gamma}^{\alpha,\beta}(f_n-f_\infty)
		> 0
\end{align*}
for any $n\geq n_0$ by Lemma \ref{Rewriting r}.
That is, we have
\begin{align*}
K_{\omega,\gamma}^{\alpha,\beta}(f_\infty)\leq0
\end{align*}
and
\begin{align}
	r_{\omega,\gamma}^{\alpha,\beta}
		\leq U_{\omega,\gamma}^{\alpha,\beta}(f_\infty). \label{120}
\end{align}
Combining \eqref{119} and \eqref{120},
\begin{align*}
	r_{\omega,\gamma}^{\alpha,\beta}
		= U_{\omega,\gamma}^{\alpha,\beta}(f_\infty).
\end{align*}
There exists $0<\lambda\leq 1$ such that
\begin{align*}
	K_{\omega,\gamma}^{\alpha,\beta}(\lambda f_\infty)=0\ \ \text{ and }\ \ 
	r_{\omega,\gamma}^{\alpha,\beta}
		\leq S_{\omega,\gamma}(\lambda f_\infty)
		= U_{\omega,\gamma}^{\alpha,\beta}(\lambda f_\infty)
		\leq U_{\omega,\gamma}^{\alpha,\beta}(f_\infty)
		= r_{\omega,\gamma}^{\alpha,\beta}.
\end{align*}
If we set $Q_{\omega,\gamma}=\lambda f_\infty$, then
\begin{align*}
	K_{\omega,\gamma}^{\alpha,\beta}(Q_{\omega,\gamma})=0\ \ \text{ and }\ \ 
	S_{\omega,\gamma}(Q_{\omega,\gamma})= r_{\omega,\gamma}^{\alpha,\beta}.
\end{align*}
\end{proof}

\begin{lemma}
Let $d \geq 1$, $2_\ast < p+1 < 2^\ast$, $\gamma > 0$, and $0 < \mu < \min\{2,d\}$.
Let $(\alpha,\beta)$ satisfies \eqref{104} and $\omega > 0$.
For any $f\in H^1(\mathbb{R}^d)$, we have
\begin{align*}
	(\mathcal{L}^{\alpha,\,\beta}-\overline{\lambda})(\mathcal{L}^{\alpha,\,\beta}-\underline{\lambda})S_{\omega,\gamma}(f)
		\leq -\frac{(p-1)\alpha\{(p-1)\alpha-2\beta\}}{p+1}\|f\|_{L^{p+1}}^{p+1}.
\end{align*}
In particular, if $f$ satisfies $\mathcal{L}^{\alpha,\,\beta}S_{\omega,\gamma}(f) = K_{\omega,\gamma}^{\alpha,\beta}(f)=0$, then it follows that
\begin{align*}
	(\mathcal{L}^{\alpha,\,\beta})^2S_{\omega,\,\gamma}(f)
		= \mathcal{L}^{\alpha,\,\beta}K_{\omega,\gamma}^{\alpha,\beta}(f)
		\leq -\overline{\lambda}\,\underline{\lambda}\,U_{\omega,\gamma}^{\alpha,\beta}(f)-\frac{(p-1)\alpha\{(p-1)\alpha-2\beta\}}{p+1}\|f\|_{L^{p+1}}^{p+1}.
\end{align*}
\end{lemma}

\begin{proof}
The simple calculation deduces the followings:
\begin{align*}
	& \mathcal{L}^{\alpha,\beta}\|f\|_{L^2}^2
		= (2\alpha-d\beta)\|f\|_{L^2}^2,\ \ \ 
	\mathcal{L}^{\alpha,\beta}\|\nabla f\|_{L^2}^2
		=\{2\alpha-(d-2)\beta\}\|\nabla f\|_{L^2}^2,\\
	& \mathcal{L}^{\alpha,\beta}\int_{\mathbb{R}^d}\frac{\gamma}{|x|^\mu}|f(x)|^2dx
		=\{2\alpha-(d-\mu)\beta\}\int_{\mathbb{R}^d}\frac{\gamma}{|x|^\mu}|f(x)|^2dx,\\
	& \mathcal{L}^{\alpha,\beta}\|f\|_{L^{p+1}}^{p+1}
		=\{(p+1)\alpha-d\beta\}\|f\|_{L^{p+1}}^{p+1}.
\end{align*}
The relations imply the desired conclusions.
\end{proof}

\begin{lemma}\label{M in G}
Let $d \geq 1$, $2_\ast < p+1 < 2^\ast$, $\gamma > 0$, and $0 < \mu < \min\{2,d\}$.
Let $(\alpha,\beta)$ satisfy \eqref{104} and $\omega > 0$.
Then,
\begin{align*}
	\mathcal{M}_{\omega,\gamma,\text{rad}}^{\alpha,\beta}
		\subset \mathcal{G}_{\omega,\gamma,\text{rad}}.
\end{align*}
holds.
\end{lemma}

\begin{proof}
We take any $\phi\in \mathcal{M}_{\omega,\gamma,\text{rad}}^{\alpha,\beta}$.
Then, we have
\begin{align*}
	\langle (K_{\omega,\gamma}^{\alpha,\beta})'(\phi),\mathcal{L}^{\alpha,\beta}\phi \rangle
		&=\mathcal{L}^{\alpha,\beta}K_{\omega,\gamma}^{\alpha,\beta}(\phi)\\
		&\leq - \overline{\mu}\underline{\mu}U_{\omega,\gamma}^{\alpha,\beta}(\phi)-\frac{(p-1)\alpha\{(p-1)\alpha-2\beta\}}{p+1}\|\phi\|_{L^{p+1}}^{p+1}
		< 0,
\end{align*}
where $\mathcal{L}^{\alpha,\beta}\phi$ denotes $\left.\frac{\partial}{\partial \lambda}\right|_{\lambda = 0}e^{\alpha\lambda}\phi(e^{\beta\lambda}\,\cdot\,)$.
There exists a Lagrange multiplier $\eta\in \mathbb{R}$ such that
\begin{align*}
	S_{\omega,\gamma}'(\phi)
		=\eta (K_{\omega,\gamma}^{\alpha,\beta})'(\phi),
\end{align*}
so we have
\begin{align*}
	0
		= K_{\omega,\gamma}^{\alpha,\beta}(\phi)
		= \mathcal{L}^{\alpha,\beta} S_{\omega,\gamma}(\phi)
		= \langle S_{\omega,\gamma}'(\phi),\mathcal{L}^{\alpha,\beta}\phi \rangle
		= \eta \langle (K_{\omega,\gamma}^{\alpha,\beta})'(\phi),\mathcal{L}^{\alpha,\beta}\phi \rangle.
\end{align*}
This implies $\eta=0$.
Therefore, we obtain $S_{\omega,\gamma}'(\phi)=0$.
We take any $\psi \in \mathcal{A}_{\omega,\gamma, \text{rad}}$.
Then, we have $K_{\omega,\gamma}^{\alpha,\beta}(\psi)=\langle S_{\omega,\gamma}'(\psi),\mathcal{L}^{\alpha,\beta}(\psi)\rangle=0$.
Thus, $S_{\omega,\gamma}(\phi)\leq S_{\omega,\gamma}(\psi)$.
\end{proof}

\begin{lemma}\label{G in M}
Let $d \geq 1$, $2_\ast < p+1 < 2^\ast$, $\gamma > 0$, and $0 < \mu < \min\{2,d\}$.
Let $(\alpha,\beta)$ satisfy \eqref{104} and $\omega > 0$.
If $\mathcal{M}_{\omega,\gamma,\text{rad}}^{\alpha,\beta}$ is not empty, then
\begin{align*}
	\mathcal{G}_{\omega,\gamma,\text{rad}}
		\subset \mathcal{M}_{\omega,\gamma,\text{rad}}^{\alpha,\beta}.
\end{align*}
\end{lemma}

\begin{proof}
We take any $\phi \in \mathcal{G}_{\omega,\gamma,\text{rad}}$.
We take $\psi\in \mathcal{M}_{\omega,\gamma,\text{rad}}^{\alpha,\beta}\subset \mathcal{G}_{\omega,\gamma,\text{rad}}$, where the last inclusion holds by Lemma \ref{M in G}.
Let $v\in H_\text{rad}^1(\mathbb{R}^d)\setminus\{0\}$ satisfy $K_{\omega,\gamma}^{\alpha,\beta}(v) = 0$.
Then, it follows that
\begin{align*}
	S_{\omega,\gamma}(\phi)
		=S_{\omega,\gamma}(\psi)
		\leq S_{\omega,\gamma}(v).
\end{align*}
In addition, $K_{\omega,\gamma}^{\alpha,\beta}(\phi) = \langle S_{\omega,\gamma}'(\phi),\mathcal{L}^{\alpha,\beta}\phi \rangle =0$.
Therefore, we obtain $\phi\in \mathcal{M}_{\omega,\gamma,\text{rad}}^{\alpha,\beta}$.
\end{proof}

\subsection*{Acknowledgements}
M.H. is supported by JSPS KAKENHI Grant Number JP19J13300.
M.I. is supported by JSPS KAKENHI Grant Number JP18H01132, JP19K14581, and JST CREST Grant Number JPMJCR1913.

\end{document}